\documentclass{amsart}
\usepackage[left=3cm,right=3cm]{geometry}
\usepackage[T1]{fontenc}
\usepackage[utf8]{inputenc}

\usepackage{amsmath}	
\usepackage{amsfonts}   
\usepackage[alphabetic]{amsrefs}
\usepackage{tikz-cd}  

\usepackage{amsthm}
\usepackage{mathtools}
\usepackage{enumerate}   

\usepackage[colorlinks]{hyperref}

\hypersetup{
	colorlinks,
	citecolor=black,
	filecolor=black,
	linkcolor=black,
	urlcolor=black
}

\newtheorem{lemma}{Lemma}[section]
\newtheorem{cor}{Corollary}
\newtheorem{rem}{Remark}
\newtheorem{thm}{Theorem}
\newtheorem{example}{Example}

\newcommand{\IN}{\mathbb{N}}
\newcommand{\IZ}{\mathbb{Z}}
\newcommand{\IQ}{\mathbb{Q}}

\newcommand{\fc}{\mathfrak{c}} 
 
\newcommand{\fL}{\mathfrak{L}} 
\newcommand{\fC}{\mathfrak{C}} 
\newcommand{\fR}{\mathfrak{R}} 
\newcommand{\fn}{\mathfrak{n}}

\newcommand{\CM}{\mathcal{M}}

\newcommand{\CW}{\mathcal{W}}
\newcommand{\CP}{\mathcal{P}}
\newcommand{\CX}{\mathcal{X}} 
\newcommand{\CY}{\mathcal{Y}} 
\newcommand{\CZ}{\mathcal{Z}} 

\newcommand{\Kl}{\mathrm{Kl}}

\newcommand{\SL}{\mathrm{SL}}             
\newcommand{\SP}{\mathrm{SP}} 
\newcommand{\GL}{\mathrm{GL}} 
\newcommand{\SO}{\mathrm{SO}} 

\DefineSimpleKey{bib}{primaryclass}{}
\DefineSimpleKey{bib}{archiveprefix}{}

\BibSpec{arXiv}{%
  +{}{\PrintAuthors}{author}
  +{,}{ \textit}{title}
  +{}{ \parenthesize}{date}
  +{,}{ arXiv }{eprint}
  +{,}{ primary class }{primaryclass}
}

\title{Bounds for Kloosterman sums for $\GL_n$}
\author{Johannes Linn }
\date{}
\thanks{The author thanks Prof.\ Valentin Blomer for helpful discussions and Dr.\ Siu Hang Man for providing parts of the code used in \cite{BM24}.}
\address{Max Planck Institute for Mathematics Bonn, Vivatsgasse 7, 53111 Bonn, Germany}
\email{linn@mpim-bonn.mpg.de}

\begin{document}
\begin{abstract}
    In this paper power saving bounds for general Kloosterman sums for all Weyl elements for $\GL_n$ for $n>2$ are proven, improving the trivial bound by Dąbrowski and Reeder. This is achieved by representing the sums in an explicit way as exponential sums and bounding these through applications of the Weil bound.
\end{abstract}
\maketitle

\section{Introduction}
Classical Kloosterman sums are defined by
\[
    S(m,n;c):=\sum_{x\in (\IZ/c\IZ)^*}e\Big(\frac{mx+n\overline{x}}{c}\Big)
\]
for $m,n\in\IZ$ and $c\in\IZ^+$. They were first encountered by Poincaré \cite{Po11} and later rediscovered by Hendrik Kloosterman while studying representations of the form $ax^2+by^2+cz^2+dt^2$ \cite{Kl26} and have become ubiquitous in Number Theory. They appear for example in Fourier coefficients of classical Poincaré series and therefore in the geometric side of relative trace formulae of Petersson-Kuznetsov type \cite{JLXY06}*{pp.\ 147-168}.
Understanding their size is crucial in many applications of these formulas.
In the classical setting, the Weil bound
\[
    |S(m,n;c)|\leq \tau(c)c^{1/2}(m,n,c)^{1/2}
\]
proven by Weil \cite{We48} for squarefree $c$ and completed by Salié \cite{Sa32} for powerful moduli
gives an optimal upper bound for the size of the sums.

Working with Poincaré series or relative trace formulae in a more general setting \cite{AB24}*{Section 5} requires the definition of more general Kloosterman sums as in \cite{Da93} and \cite{DR98}.

Let $G$ be a reductive group, $U$ a maximal unipotent subgroup, $U^-$ the opposite unipotent subgroup, and $T$ a maximal torus. Let $N$ be the normalizer of $T$ in $G$ and let $W:=N/T$ be the Weyl group. For a Weyl element $w\in W$ define
\[
    U_w:=U\cap (w^{-1}U^{-}w).
\]

Let $\Gamma\subseteq G(\IZ_p)$ be a finite index subgroup.
For a diagonal matrix $\fc$ with determinant one and a Weyl element $w\in W$ write $\fn:=\fc\cdot w$ and define
\[
    C(\fn):=U(\mathbb Q_p)\fn U(\mathbb Q_p)\cap \Gamma
\]
and the Kloosterman set
\[
    X(\fn):=U(\mathbb Z_p)\backslash C(\fn)/U_w(\mathbb Z_p)
\]
together with the two projections $u:X(\fn)\rightarrow U(\IZ_p)\backslash U(\IQ_p)$ and $u^\prime:X(\fn)\rightarrow U_w(\IQ_p)/U_w(\IZ_p)$.

Let $\psi$ and $\psi^\prime$ be characters on $U(\IQ_p)$ trivial on $U(\IZ_p)$. The generalized (local) Kloosterman sum is defined as
\[
    \Kl^{\Gamma}_p(\psi,\psi^\prime, \fc\cdot w):=\sum_{u\fn u^\prime\in X(\fn)}\psi(u)\psi^\prime(u^\prime).
\]
If $\Gamma=G(\IZ_p)$, it will be suppressed in the notation.
\begin{rem}
\label{rem:1}
   The Kloosterman sum can also be defined in a slightly different way by using the set
    \[
        X^\prime(\fn):=U_{w^{-1}}(\mathbb Z_p)\backslash C(\fn)/U(\mathbb Z_p).
    \]
    In applications, the definition depends on whether one parametrizes the Bruhat cells by $UTwU_w$ or $U_{w^{-1}}TwU$. Inverting all elements switches between the two definitions and leads to
    \[
        \Kl^{\Gamma}_p(\psi,\psi^\prime, \fc \cdot w)={\Kl^\prime}^{\Gamma}_p(\overline\psi^\prime,\overline\psi, (w^{-1}\fc w)\cdot w^{-1}).
    \]  
\end{rem}

\begin{example}
    The classical Kloosterman sum can be recovered by setting
    \[
        G=\GL_2, \fn=\begin{psmallmatrix}
        0 & -c^{-1}\\ c&0
    \end{psmallmatrix}, \psi\big(\begin{psmallmatrix}
        1&x\\&1
    \end{psmallmatrix}\big)=\Psi(mx), \text{and } \psi^\prime\big(\begin{psmallmatrix}
        1&x\\&1
    \end{psmallmatrix}\big)=\Psi(nx),
    \]
    where $\Psi$ is the standard additive character on $\IQ_p/ \IZ_p$.
\end{example}

Obtaining bounds for generalized Kloosterman sums is surprisingly more complicated than for classical ones. Even just computing the size of $X(\fn)$ and therefore obtaining a trivial bound for the sum is already a deep result by Dąbrowski and Reeder \cite{DR98}*{Proposition 3.4}. For the special Weyl element
{$\begin{psmallmatrix}
    &1\\ I_n &
\end{psmallmatrix}$}
the general Kloosterman sum turns out to be a Hyper-Kloosterman sum \cite{Fr87}*{Theorem B}
and can be bounded by algebro-geometric means. Besides this, non-trivial bounds are known for $\GL_3$ see \cites{BFG88,DF97,St87}, $\GL_4$ see \cites{GSW21,BM24}, $\SP_4$ see \cite{Ma22}, $\SO_{3,3}$ and $\SO_{4,2}$ see \cite{Mu24}, and $\GL_n$ for the long Weyl element see \cites{Mi24,BM24} and the special Weyl element
{$w_*:=\begin{psmallmatrix}
        &&1\\
        &I_{n-2}&\\
        1&&
    \end{psmallmatrix}$} see \cite{BM24}.

The three main approaches in these cases all involve a parametrization or stratification of the Kloosterman set. In \cites{BFG88,Mu24,Fr87} Plücker coordinates are used. This approach seems to fail for larger $n$ as the relations become too complicated to manage.
In \cites{St87,Mi24} the action of the torus on the Kloosterman set by conjugation is exploited to obtain a decomposition of the sum. This approach is also useful for other congruence subgroups besides $\GL_n(\IZ_p)$ see \cite{Mi24}*{Section 5} as the torus action preserves congruence subgroups, but it also requires explicit calculations, which can get very complicated in general. Lastly, Blomer and Man \cite{BM24} use the stratification of Dąbrowski and Reeder \cite{DR98}*{Section 3} to obtain an explicit parametrization of the Kloosterman sum, which can be bounded using a clever application of the Weil bound for $\GL_2$. 

This paper uses the last approach to obtain a parametrization and bounds for the Kloosterman sums for $G=\GL_{N+1}$ for all $N\geq 2$, all admissible Weyl elements, and $\Gamma=G(\IZ_p)$ and $\Gamma=\Gamma_0(q)$.

The main new ingredients are:

A diagram associated to a Weyl element, which is used to make the stratification of Dąbrowski and Reeder \cite{DR98}*{Section 3} explicit for all Weyl elements with two blocks, and to apply multiple Weil bounds to the parametrized sum at the same time.

And an induction on the number of blocks of a Weyl element, which enables the treatment of more than one specific Weyl element at a time.

The main result is the following bound that improves on the trivial bound by Dąbrowski and Reeder by a power saving of $\frac{1}{4l(w)}$, which is proven in Section \ref{sec:4}.

\begin{thm}
\label{thm:5}
    Let $\fn=\fc\cdot w\in\GL_{N+1}$ with exponent vector $r$. For an admissible Weyl element $w$ with length $l(w)$, characters $\psi$ and $\psi^\prime$, and all $\varepsilon>0$ the following holds
    \[
        \Kl_p(\psi,\psi^\prime,\fn)\ll_\varepsilon C^{l(w)/2}\cdot \Big(\prod_{1\leq k\leq N}p^{r_k}\Big)^{1-\frac{1}{4l(w)}+\varepsilon}\ll C^{\frac{N(N+1)}{4}}\cdot \Big(\prod_{1\leq k\leq N}p^{r_k}\Big)^{1-\frac{1}{2N(N+1)}+\varepsilon},
    \]
    with
    \[
        C:=\max_{1\leq j\leq N}\min(|\psi_j|_p^{-1/2},p^{r_j/2}).
    \]
    Also 
    \[
        \Kl_p^{\Gamma_0(q)}(\psi,\psi^\prime,\fn)\ll_\varepsilon C^{l(w)/2}\cdot \Big(\prod_{1\leq k\leq N}p^{r_k}\Big)^{1-\frac{1}{4l(w)}+\varepsilon}\ll C^{\frac{N(N+1)}{4}}\cdot \Big(\prod_{1\leq k\leq N}p^{r_k}\Big)^{1-\frac{1}{2N(N+1)}+\varepsilon}
    \]
    holds for $q$ a power of $p$.
\end{thm}
\begin{rem}
    For all Weyl elements besides $w_*$, {$\begin{psmallmatrix}
        &1\\ I_N &
        \end{psmallmatrix}$}, and the long Weyl element this represents the first non-trivial bound, and for the long Weyl element it improves the power saving obtained in \cite{BM24} by a factor of $4$.
        
        In terms of the optimal bound: Numerical computations indicate that using the method employed in this paper a saving of $\frac{1}{o(N^2)}$ cannot be obtained. It is also known by \cite{BM24}*{Corollary 3} that the saving cannot be $\frac{1}{o(N)}$ in general.
\end{rem}

\begin{rem}
    The bound in Theorem \ref{thm:5} has a large dependency on the character. As explained in section \ref{subsec:43} a simple dependency as in the $\GL_2$ case can not be obtained if one wishes to optimize the bounds. This can already be seen in \cite{BFG88}*{Property 4.10 (a)} as the sum grows by a factor $p$ and not $p^{1/2}$ if the $\psi_i$ are divisible by $p$. Theorem \ref{thm:6} gives a nice dependency at the cost of a worse power saving.  
\end{rem}
\begin{thm}
\label{thm:6}
With the same notation as in Theorem \ref{thm:5} the following holds
\[
        \Kl_p(\psi,\psi^\prime,\fn)\ll_\varepsilon C\cdot \Big(\prod_{1\leq k\leq N}p^{r_k}\Big)^{1-\frac{1}{2N\cdot l(w)}+\varepsilon}\ll C\cdot \Big(\prod_{1\leq k\leq N}p^{r_k}\Big)^{1-\frac{1}{N^2(N+1)}+\varepsilon},
    \]
    with
    \[
        C:=\max_{1\leq j\leq N}\min(|\psi_j|_p^{-1/2},p^{r_j/2}).
    \]
    Also 
    \[
        \Kl_p^{\Gamma_0(q)}(\psi,\psi^\prime,\fn)\ll_\varepsilon C\cdot \Big(\prod_{1\leq k\leq N}p^{r_k}\Big)^{1-\frac{1}{2n\cdot l(w)}+\varepsilon}\ll C\cdot \Big(\prod_{1\leq k\leq N}p^{r_k}\Big)^{1-\frac{1}{N^2(N+1)}+\varepsilon}
    \]
    holds for $q$ a power of $p$.
\end{thm}

\begin{rem}
     All proofs in this paper do not use that the local field is $\IQ_p$ and the two main ingredients the stratification by Dąbrowski and Reeder and the Weil bound also hold over general non-archimedean local fields. Theorem \ref{thm:5} and \ref{thm:6} and all other results in this paper therefore also hold over general non-archimedean local fields.
\end{rem}

Both theorems uses the following notation, which will be fixed for the rest of the paper.
The torus $T$ denotes the diagonal matrices and $U$ denotes the upper triangular matrices with diagonal equal to $1$.
Two characters $\psi$ and $\psi^\prime$ on $U(\IQ_p)$ which are trivial on $U(\IZ_p)$ are considered and written as
\[
    \psi(u)=\Psi(\sum_{i=1}^N \psi_i u_{ii+1}) \text{ and } \psi^\prime(u)=\Psi(\sum_{i=1}^N \psi^\prime_i u_{ii+1}),
\]
where $u\in U$ and $\Psi$ is the standard additive character of $\IQ_p$ trivial on $\IZ_p$.

For a modulus $\fc=\mathrm{diag}(p^{-r_1},p^{r_1-r_2},...,p^{r_{N-1}-r_N},p^{r_N})$ let $r=(r_1,...,r_N)\in \IZ_{\geq 0}^N$ be the exponent vector.
Each modulus $\fc$ defining a non-zero Kloosterman sum can be written like this up to a factor in $T(\IZ_p)$, which can always be absorbed into the characters.

Not all $(N+1)!$ Weyl elements in $W$ are considered because only Weyl elements of the form
\[
    w=\begin{pmatrix}
        &&&& I_{k_1}\\
        &&&I_{k_2}&\\
        &&...&&\\
        &I_{k_{n-1}}&&&\\
        I_{k_n}&&&&
    \end{pmatrix},
\]
with $I_d$ being the $(d\times d)$-identity matrix, lead to well-defined Kloosterman sums \cite{Fr87}*{p.\ 175}. These elements are called admissible and will be represented by matrices with determinant one, which might not be indicated.

The proof of Theorem \ref{thm:5} uses a parametrization of the Kloosterman sum. To state it some preparation is needed. Let $w\in W$ be as above. Let $\kappa_l:=\sum_{a=1}^l k_a$ and define
\[
    I_{l}:=\{(i,j):1\leq i\leq \kappa_l\leq j\leq \kappa_{l+1}-1\} \text{ and } I_w:=\bigcup_{l=1}^{n-1} I_l.
\]
Let $l(w)$ denote the length of $w$. Then, $|I_w|=l(w)$. The $I_l$ can be thought of in two ways. On the one hand, they parametrize the positive roots $\alpha$ for which $w\alpha$ is negative. On the other hand, $w$ can be factorized as
\[
    w=w_{n-1}\cdot w_{n-2}\cdot ...\cdot w_2\cdot w_1,
\]
with
\[
    w_l=\begin{pmatrix}
    I_{N+1-\kappa_{l+1}}&&\\
    &&I_{\kappa_l}\\
    &I_{k_{l+1}}&
    \end{pmatrix}\ ,
\]
and $I_l$ can be viewed as attaching coordinates to $*$ in 
\[
U_{w_l}=\left\{\begin{pmatrix}
I_{N+1-\kappa_{l+1}}&&\\
&I_{k_{l+1}} &  *\\
&&I_{\kappa_l}
\end{pmatrix}\right\}\ .
\]

For an exponent vector $r$ define
\[
    \CM_w(r):=\{\underline{m}=(m_{i,j})_{(i,j)\in I_w}\in \IN^{l(w)}:\forall 1\leq l\leq N:\sum_{i\leq l\leq j}m_{i,j}=r_l\}, 
\]
For $(i,j)\in I_l$ define
\[
    C_{i,j}:=C_{i,j}(\underline{m},w)=p^{\sum_{a=1}^i m_{a,j}+\sum_{a=j+1}^{\kappa_{(l+1)}-1}m_{i,a}}.
\]
Finally, define
\[
    C_w(\underline{m}):=\{\underline{c}=(c_{i,j})_{(i,j)\in I_w}\in\IZ^{l(w)}:\forall (i,j)\in I_w: 0\leq c_{i,j}<C_{i,j}\text{ and } (c_{i,j},p^{m_{i,j}})=1\},
\]
which is the range of summation of the parametrized Kloosterman sum. From the size of $C_w(\underline{m})$ one can immediately see the trivial bound by Dąbrowski and Reeder.

For a summation variable $c_{i,j}$ define $d_{i,j}\in\IZ$ such that $c_{i,j}\cdot d_{i,j}\equiv 1 \mod C_{i,j}$ if $m_{i,j}>0$, and $d_{i,j}=0$ if $m_{i,j}=0$, which enables the treatment of $m_{i,j}>0$ and $m_{i,j}=0$ at the same time.

The second main result of this paper is the following parametrization of the Kloosterman sum, which is proven in Section \ref{sec:3}.

\begin{thm}
\label{thm:4}
    Let $\fn=\fc\cdot w\in\GL_{N+1}$ with exponent vector $r$ and
        \[
            w=\begin{pmatrix}
                &&&&I_{k_1}\\
                &&&I_{k_2}&\\
                &&...&&\\
                &I_{k_{n-1}}&&&\\
                I_{k_n}&&&&
            \end{pmatrix}.
        \]
        The Kloosterman sum can be parametrized by
        \[
            \Kl_p(\psi,\psi^\prime,\fn)=\sum_{\underline m\in\CM_{w}(r)}\Kl_p(\underline m,\psi,\psi^\prime,w),
        \]
        with $\Kl_p(\underline m,\psi,\psi^\prime,w)$ being defined as
        \begin{align*}
            \sum_{\underline c\in C_{w}(\underline m)}\Psi\big(
            \sum_{q=1}^{n-1}&\big(
            \sum_{i=1}^{\kappa_q-1}\psi_i\sum_{j=\kappa_q}^{\kappa_{(q+1)}-1} c_{i+1,j}d_{i,j}p^{-m_{i,j}+\sum_{l=\kappa_q}^{j-1}(m_{i+1,l}-m_{i,l})+\sum_{l=\kappa_{(q+1)}}^{N}(m_{i+1,l}-m_{i,l})}\\
                 +&\ \psi_{\kappa_q}c_{\kappa_q,\kappa_q}p^{-m_{\kappa_q,\kappa_q}+\sum_{l=\kappa_{(q+1)}}^{N}(m_{\kappa_q+1,l}-m_{\kappa_q,l})}\\
                 +&\sum_{j=\kappa_q+1}^{\kappa_{(q+1)}-1}\psi_j\sum_{i=1}^{\kappa_q} c_{i,j-1}d_{i,j}p^{-m_{i,j}+\sum_{l=i+1}^{\kappa_q}(m_{l,j-1}-m_{l,j})+\sum_{l=\kappa_{(q+1)}}^{N}(m_{j+1,l}-m_{j,l})}\\
                 +&\ \psi^\prime_{N+1-\kappa_q}\sum_{i=1}^{\kappa_{(q-1)}+1}  c_{i,\kappa_{(q+1)}-1}d_{i,\kappa_{(q-1)}}p^{-m_{i,\kappa_{(q+1)}-1}+\sum_{j=1}^{i-1}(m_{j,\kappa_{(q-1)}}-m_{j,\kappa_{(q+1)}-1})}\big)\big).
        \end{align*}
        Here $\kappa_0:=0$ and $d_{i,j}:=1$ if $i>j$.
\end{thm}

At first glance, this might seem very cumbersome but the later sections will demonstrate a more compact way to think about this parametrization, which is stated in Corollary \ref{cor:2}.

In specific applications, for example when working with a fixed Weyl element, Theorem \ref{thm:4} can be used to obtain sharper bounds than Theorem \ref{thm:5} and is therefore of independent interest.

\section{Setting}
This paper uses the setting of \cite{DR98} specialized to the case $G=\GL_{N+1}$ and the field $\IQ_p$.
This section is intended to be used for future reference. For more detailed definitions, explanations, and proofs the reader is referred to \cite{DR98}.

Let $\mathfrak g$ be the Lie algebra associated to $\GL_{N+1}$. Let $\Delta:=\{\alpha_i:1\leq i\leq N\}$ be the set of simple roots and let
\[
    \Phi^+=\{\alpha_{i,j}:=\sum_{l=i}^j \alpha_i:1\leq i\leq j\leq N\}
\]
be the set of positive roots, $\Phi$ the set of roots, and $\check{\Phi}$ the set of coroots.
For each positive root $\alpha$ there is an embedding 
\[
    \phi_\alpha:\SL_2\rightarrow \GL_{N+1}.
\]
For each simple root $\alpha_i$ let $s_i:=s_{\alpha_i}$ be the reflection associated to that root. In the case of $\GL_{N+1}$ this acts on the other simple roots by
\[
    s_{\alpha_i}(\alpha_j)=\begin{cases} -\alpha_j & \text{ if }  i=j\\
    \alpha_i+\alpha_j & \text{ if } |i-j|=1\\
    \alpha_j & \text{ else} 
    \end{cases}
\]
and can be written using the above embeddings as
\[
    s_{i}=\phi_{\alpha_i}\big(\begin{psmallmatrix}
    & -1\\1&
\end{psmallmatrix}\big)\ .
\]

Let $\check{X}$ be the set of cocharacters of $T$. For such a cocharacter $\lambda=-\sum_{\beta\in\Delta}r_\beta\check{\beta}$ with $r_\beta\in \IZ_{\geq 0}$ define the height by
\[
    \mathrm{ht}(\lambda):=\sum_{\beta\in\Delta} r_\beta.
\]
The height is needed for the computation of the size of the Kloosterman set.
A coroot $\check{\beta}$ can be embedded into $T$ by
\[
    \check{\beta}\mapsto \phi_\beta\Big(\begin{pmatrix}
        p&\\&p^{-1}
    \end{pmatrix}\Big),
\]
which induces a natural embedding from $\check{X}$ into $T$.

For a Weyl element $w\in W$ define
\[
    R(w):=\{\beta\in\Phi^+: w\beta\in-\Phi^+\}.
\]
A minimal way to write a Weyl element $w$ as a product of simple reflections $s_{i}$ is called a reduced expression. 
If $w=s_{i_1}\circ...\circ s_{i_k}$ is such a reduced expression the above set can be written as
\[
    R(w^{-1})=\{\gamma_j:=s_{i_1}\circ...\circ s_{i_{j-1}}(\alpha_{i_j}):1\leq j\leq k\}.
\]
The $\gamma_j$ will be the building blocks in the parametrization for the Kloosterman sum. For an admissable Weyl element as in Theorem \ref{thm:4} the set can be written as
\[
    R(w^{-1})=\{\alpha_{i,j} : (i,j)\in I_w\} \ ,
\]
where each $I_l$ corresponds to the $\gamma_i$ obtained from the factor $w_l$.

For a positive root $\alpha$ and $a\in\IQ_p$ define
\[
    b_\alpha(a):=\begin{cases}
        \phi_\alpha\big(\begin{psmallmatrix}
            & -1\\ 1& a
        \end{psmallmatrix}\big)& \text{ if } a\in\IZ_p\ ,\\
        \phi_\alpha\big(\begin{psmallmatrix}
            c^{-1}&\\p^m&c
        \end{psmallmatrix}\big)&\text{ if } a=cp^{-m}, c\in\IZ_p^*,m>0\ ,
    \end{cases}
\]
and
\[
    b_{\underline\beta}(\underline a):=\prod_{i=1}^l b_{\beta_i}(a_i)
\]
for $\underline \beta=(\beta_1,...,\beta_l)\in (\Phi^+)^l$ and $\underline  a=(a_1,...,a_l)\in\IQ_p^l$.

For the rest of the section fix a Weyl element $w$ and a reduced expression $w=s_{i_1}\circ...\circ s_{i_l}$. The upcoming definitions will implicitly depend on this Weyl element. In the rest of the paper, the current Weyl element will always be clear.

Let $\underline \beta=(\alpha_{i_1},...,\alpha_{i_l})$. Then, \cite{DR98}*{Proposition 3.1} implies that 
\[
    b_{\underline \beta}: \IQ_p^l\rightarrow U(\IZ_p)\backslash\big((U(\IQ_p)wT(\IQ_p)U(\IQ_p))\cap \GL_{N+1}(\IZ_p)\big)
\]
is injective.
\begin{rem}
    \cite{DR98}*{Proposition 3.1} states the injectivity for a right quotient but the proof can be repeated in the same way simply using the other equality in \cite{DR98}*{(1.7)}.
\end{rem}
For $a\in\IQ_p$ define
\[
    \mu(a):=\begin{cases}
        0 &\text{ if } a\in\IZ_p\ ,\\
        -v_p(a)&\text{ else}  ,
    \end{cases}
\]
where $v_p$ is the $p$-adic valuation.
For $\underline m=(m_1,...,m_l)\in\IZ_{\geq0}^l$ set
\[
    Y(\underline m):=\{\underline a=(a_1,...,a_l)\in\IQ_p:\forall 1\leq i\leq l:\mu(a_i)=m_i\}.
\]
\cite{DR98}*{Proposition 3.2} implies the existence of a right group action $*$ of $U(\IZ_p)$ on $Y(\underline m)$ such that $b_{\underline\beta}$ is equivariant under this action, i.e.
\[
    b_{\underline\beta}(u*\underline a)=b_{\underline\beta}(\underline a)u.
\]
\begin{rem}
    Again \cite{DR98}*{Proposition 3.2} defines a left group action but as in \cite{BM24}*{p.\ 1180} this can be modified to a right action.
\end{rem}
For the exact definition of the action and more details the reader is referred to \cite{DR98}*{pp.\ 241-242} and \cite{BM24}*{p.\ 1180}.

Recall that
\[
    R(w^{-1})=\{\gamma_j:=s_{i_1}\circ...\circ s_{i_{j-1}}(\alpha_{i_j}):1\leq j\leq k\}\ .
\]

For any $\lambda\in \check{X}$ define
\[
    M_\lambda = \{\underline m\in\IZ_{\geq 0}^l: \lambda=-\sum_{j=1}^lm_j\check{\gamma}_j\}\ ,
\]
which is empty if $\lambda$ is not a sum of coroots $\check{\gamma}_j$, and define
\[
    Y_\lambda = \bigcup_{\underline m\in M_\lambda} Y_{\underline m}\ .
\]
This enables the statement of the two main results needed from this section.

\begin{lemma}[{\cite{DR98}*{Proposition 3.3}}]
\label{lem:1}
    Let $\lambda\in\check{X}$, $\fn=\lambda w$, and let $w$ and $\underline\beta$ be as before. Then $ b_{\underline \beta}$ descends to a bijection between $Y_\lambda/U_w(\IZ_p)$ and $X(\fn)$.
\end{lemma}
Furthermore, the exact size of the Kloosterman set is given by the following lemma.
\begin{lemma}[\cite{DR98}*{Proposition 3.4}]
\label{lem:2}
    With $\underline m\in\IZ_{\geq0}^l$ and the same notation as before
    \[
        |Y(\underline m)/U_w(\IZ_p)|=p^{\mathrm{ht}(\lambda)}(1-p^{-1})^{\kappa(\underline m)},
    \]
    where $\kappa(\underline m)$ is the number of non-zero entries in $\underline m$.
\end{lemma}

The proof of Theorem \ref{thm:5} uses one more convention not used in \cite{DR98} and \cite{BM24}, which is necessary to deal with all Weyl elements using the same framework.

For an admissible Weyl element only the following reduced expression is considered. It is defined inductively over the number of blocks $n$ by
 \[
    \begin{pmatrix}
        &&&&I_{k_1}\\
        &&&I_{k_2}&\\
        &&...&&\\
        &I_{k_{n-1}}&&&\\
        I_{k_n}&&&&
    \end{pmatrix}=\begin{pmatrix}
        I_{k_1}&&&&\\
        &&&&I_{k_2}\\
        &&&...&\\
        &&I_{k_{n-1}}&&\\
        &I_{k_n}&&&
    \end{pmatrix}(s_{k_1}\circ...\circ s_N)\circ...\circ(s_{1}\circ...\circ s_{N-k_1+1}).
\]
In other words, first, the rightmost block is moved element by element to the left and then this is repeated with the now rightmost block, and so on. For example the element
{$
    \begin{psmallmatrix}
        &&I_2\\
        &I_3&\\
        1&&
    \end{psmallmatrix}
$}
for $\GL_6$ would have the reduced expression
\[
    (s_5\circ s_4\circ s_3)\circ(s_2\circ s_3 \circ s_4 \circ s_5 \circ s_1\circ s_2 \circ s_3 \circ s_4).
\]

\section{Bruhat decomposition and System of representatives}
\label{sec:3}
    \subsection{Structure of the Proof}
        The first step is to understand the Bruhat decomposition of $b_{\underline{\alpha}}(\underline a)$ for an element $\underline a\in Y(\underline m)$. Assume this decomposition is given by $L\cdot C\cdot R$. Knowing all the entries of $L$ and $R$ one can compute the characters $\psi(L)$ and $\psi^\prime(R)$ and write them explicitly in terms of $\underline a$. The entries of $R$ can also be used to find a system of representatives for the right action of $U_w(\IZ_p)$ on the image under $b_{\underline \alpha}$ of $Y(\underline m)$. Therefore, knowing all entries of $L$ and $R$ completely enables the statement of the sum in an explicit way.

        Understanding the entries of $L$ and $R$ completely for arbitrary Weyl elements and writing them as a closed expression seems too complicated.
        This problem will be overcome in two ways.
        Firstly, $\psi(L)$ (resp.\ $\psi^\prime(R)$) does not depend on all entries of $L$ (resp.\ R) but only on the entries of the first off-diagonal.
        Secondly, to find a system of representatives the entries of $R$ are not needed exactly as long as the uncertainty can be absorbed into the $U_w(\IZ_p)$ action.

        To enable the treatment of all Weyl elements Theorem \ref{thm:5} and Theorem \ref{thm:4} are first proven for ``simple'' Weyl elements, namely Weyl elements consisting of exactly two blocks. For those Weyl elements, the entries of $L$ and $R$ are computed exactly and a system of representatives for the right $U_w(\IZ_p)$ action is found. After that, an induction on the number of blocks is performed to treat all Weyl elements, where the entries of $L$ and $R$ are no longer completely understood but enough to compute the characters and find a system of representatives.

    \subsection{Weyl elements with only two blocks}
        Recall that $b_{\underline \alpha}(\underline a)$ was defined as a product of the $b_\alpha(a)$. For a single factor, the Bruhat decomposition is given by
        \[
            b_\alpha(a) = \phi_\alpha\left(\begin{pmatrix}
                1 & d p^{-m}\\ & 1
            \end{pmatrix}
            \begin{pmatrix}
                 &  -p^{-m}\\ p^m& 
            \end{pmatrix}
            \begin{pmatrix}
                1 & c p^{-m}\\ & 1
            \end{pmatrix}\right)
        \]
        for $a=c$ and $d=m=0$ if $a\in\IZ_p$, and $a=c\cdot p^{-m}$, $c\in\IZ_p^*$ and $d=c^{-1}$ if $v_p(a)<0$. As in \cite{BM24}*{Lemma 3} the $d$ acts as a formal inverse of $c$ to treat both cases at the same time.

        As a running example consider the Weyl element
        \[
            w_0:=\begin{pmatrix}
                &&&1&\\
                &&&&1\\
                1&&&&\\
                &1&&&\\
                &&1&&
            \end{pmatrix} = s_2\circ s_3\circ s_4\circ s_1\circ s_2\circ s_3,
        \]
        which is used to illustrate the abstract formulas and arguments in the following parts.

        To state the entries of $L$ and $R$ in a compact way they are written in terms of paths in the following diagram consisting of the roots
        
        \begin{align*}
            \gamma_i\in R(w^{-1})&=\{\alpha_2,s_2(\alpha_3),s_2s_3(\alpha_4),s_2s_3s_4(\alpha_1,s_2s_3s_4s_1(\alpha_2),s_2s_3s_4s_1s_2(\alpha_3))\}\\
            &=\{\alpha_{2,2},\alpha_{2,3},\alpha_{2,4},\alpha_{1,2},\alpha_{1,3},\alpha_{1,4}\}\ .
        \end{align*}
        
        They are grouped into rows by their smallest summand, meaning $\alpha_{i,j}$ and $\alpha_{i,l}$ are in the same row. They are also ordered according to the usual ordering of the $\gamma_i$. In the case of $w_0$, the diagram looks like
        
        \begin{center}
        \begin{tikzcd}
            & & \\
            2,2 & 2,3 & 2,4 \\
            1,2 & 1,3 & 1,4. \\
            & & 
        \end{tikzcd} 
        \end{center}

        Here the $\alpha$ is omitted from each entry. Therefore, $(2,2)$ stands for $\alpha_{2,2}=\gamma_1$, $(2,3)$ stands for $\alpha_{2,3}=\gamma_2$ and so on.

        The diagram also contains edges between entries $(i,j)$ and $(i,j+1)$, and $(i,j)$ and $(i-1,j)$. These correspond to simple reflections acting on the roots.
        
        For $w_0$ this leads to:
        
        \begin{center}
        \begin{tikzcd}
            & & \\
            2,2 \arrow[r,dash] \arrow[d,dash]& 2,3 \arrow[r,dash] \arrow[d,dash] & 2,4 \arrow[d,dash] \\
            1,2 \arrow[r,dash] & 1,3 \arrow[r,dash] & 1,4. \\
            & & 
        \end{tikzcd}
        \end{center}
        Here the edge from $(2,2)$ to $(2,3)$ corresponds to the reflection $s_3$.
        
        The formulas for the entries of $L, C$, and $R$ contain sums over paths from one vertex in the diagram to another of minimal length. These can be seen as Weyl elements sending $\gamma_i$ to $\gamma_j$ of minimal length. As they are minimal, each path only contains edges going in one of two directions. For such a path $\CP$ the vertices are grouped into sets according to the direction of the adjacent edges. The set $\CP_{ld}$ (resp.\ $\CP_{lr}$,...) contains all vertices with adjacent edges going left and down (resp.\ left and right,...). The treatment of the first and last vertex in the path will be specified in the definitions.
        
        As an example, look at the path $(1,4),(1,3),(2,3),(2,2)$. Here $(2,3)$ has an edge to the left and one downwards and is in $\CP_{ld}$. The sets of $(1,4)$ and $(2,2)$ depend on the specification and for example, $(1,4)$ is in $\CP_{ld}$ if it is treated with an incoming edge from below.

        These definitions enable the statement of the formulas for the Bruhat decomposition of a Weyl element 
        \[
            w_{k,n+1-k}:=\begin{pmatrix}
                &I_k\\
                I_{n+1-k}
            \end{pmatrix}
        \]
        for $\GL_{n+1}$ corresponding to the following diagram:

        
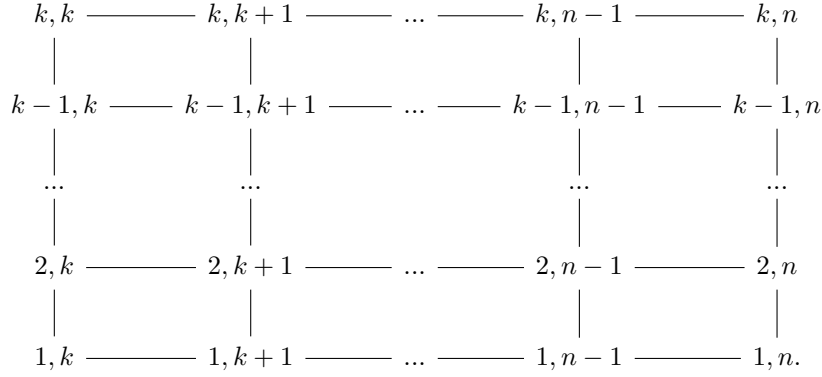
\begin{center}
        \begin{tikzcd}
            & & & &\\
            k,k \arrow[r,dash] \arrow[d,dash]& k,k+1 \arrow[r,dash] \arrow[d,dash] & ... \arrow[r,dash] & k,n-1 \arrow[r,dash] \arrow[d,dash] &k,n\arrow[d,dash]\\
            k-1,k \arrow[r,dash] \arrow[d,dash]& k-1,k+1 \arrow[r,dash] \arrow[d,dash] & ... \arrow[r,dash] & k-1,n-1 \arrow[r,dash] \arrow[d,dash] &k-1,n\arrow[d,dash]\\
            ...  \arrow[d,dash]& ...  \arrow[d,dash] & & ...  \arrow[d,dash] &...\arrow[d,dash]\\
            2,k \arrow[r,dash] \arrow[d,dash]& 2,k+1 \arrow[r,dash] \arrow[d,dash] & ... \arrow[r,dash] & 2,n-1 \arrow[r,dash] \arrow[d,dash] &2,n\arrow[d,dash]\\
            1,k \arrow[r,dash]& 1,k+1 \arrow[r,dash] & ... \arrow[r,dash] & 1,n-1 \arrow[r,dash] &1,n.\\
            & & & &
        \end{tikzcd}
        \end{center}

        \begin{thm}
        \label{thm:1}
            Let $w=w_{k,n+1-k}$ be as before. Write $a_{i,j}=c_{i,j}p^{-m_{i,j}}$ with $m_{i,j}=0$ if $a_{i,j}\in\IZ_p$, and $c_{i,j}\in\IZ_p^*$ if $v_p(a_{i,j})<0$. Let $d_{i,j}:=0$ in the first case and $d_{i,j}=c_{i,j}^{-1}$ in the second case. The following holds for the Bruhat decomposition $L\cdot C\cdot R$ of $b_{\underline \alpha}(\underline a)$.

            Let $\CX_{i,j}$ be the set of paths  of minimal length from $(1,k+i-1)$ to $(j,n)$, where the first vertex is treated as a vertex with an incoming edge from below and the last vertex as a vertex with an outgoing edge to the right. 
            Then the right factor is given by
            \[
                R=\begin{pmatrix}
                    I_{n+1-k} &X\\
                    &I_{k}
                \end{pmatrix},
            \]
            with
            \[
                X_{i,j}=(-1)^{i+n+1-k}\sum_{\CP\in \CX_{i,j}}\prod_{a,b\in\CP_{rd}}c_{a,b}\prod_{a,b\in\CP_{lu}}d_{a,b}\cdot p^{-\sum_{a,b\in\CP}m_{a,b}}.
            \]

            Let $\CY_{i,j}$ be the set of paths from $(i,k)$ to $(j,n)$  of minimal length, where the first vertex is treated as a vertex with an incoming edge from the left and the last vertex as a vertex with an outgoing edge to the right.
            
            Let $\CZ_{i,j}$ be the set of paths from $(1,j+k-1)$ to $(i,k)$  of minimal length, where the first vertex is treated as a vertex with an incoming edge from below and the last vertex as a vertex with an outgoing edge to the left.
            
            Let $\CW_{i,j}$ be the set of paths from $(1,j+k-1)$ to $(k,i+k-1)$  of minimal length, where the first vertex is treated as a vertex with an incoming edge from below and the last vertex as a vertex with an outgoing edge upwards.
            
            Then the left factor is given by
            \[
                L=\begin{pmatrix}
                    Y&Z\\&W
                \end{pmatrix},
            \]
            where $Y$ is a unipotent $k\times k$ upper triangular matrix, $W$ is a unipotent $(n+1-k)\times (n+1-k)$ upper triangular matrix, and $Z$ is a $k\times (n+1-k)$ matrix, with entries given by
            \[
                Y_{i,j}=\sum_{\CP\in \CY_{i,j}}\prod_{a,b\in\CP_{rd}}c_{a,b}\prod_{a,b\in\CP_{lu}}d_{a,b}\cdot p^{-\sum_{a,b\in\CP}m_{a,b}+\sum_{l=k}^n m_{j,l}}
            \]
            for $1\leq i\leq j\leq k$ and
            \[
                Z_{i,j}=\sum_{\CP\in \CZ_{i,j}}\prod_{a,b\in\CP_{ru}}c_{a,b}\prod_{a,b\in\CP_{ld}}d_{a,b}\prod_{a,b\in\CP_{lr}}(c_{a,b}d_{a,b}-1)\cdot p^{-\sum_{a,b\in\CP_{lr}}m_{a,b}+\sum_{a,b\in\CP_{du}}m_{a,b}-\sum_{l=1}^{k} m_{l,j+k-1}}
            \]
            for $1\leq i\leq k$, $1\leq j\leq n+1-k$ and
            \[
                W_{i,j}=\sum_{\CP\in \CW_{i,j}}\prod_{a,b\in\CP_{ru}}c_{a,b}\prod_{a,b\in\CP_{ld}}d_{a,b}\prod_{a,b\in\CP_{lr}}(c_{a,b}d_{a,b}-1)\cdot p^{-\sum_{a,b\in\CP_{lr}}m_{a,b}+\sum_{a,b\in\CP_{du}}m_{a,b}-\sum_{l=1}^{k} m_{l,j+k-1}}
            \]
            for $1\leq i\leq j\leq n+1-k$.
            All other entries of $L$ are zero. Notice that $Y_{i,i}$ and $W_{i,i}$ are always $1$.

            The central factor is given by
            \[
                C=\begin{pmatrix}
                    &C_1\\
                    C_2&
                \end{pmatrix},
            \]
            where $C_1$ is a diagonal $k\times k$ matrix and $C_2$ is a diagonal $(n+1-k)\times (n+1-k)$ matrix. Their entries are given by
            \[
                ({C_1})_{i,i}=(-1)^{n+1-k}p^{-\sum_{l=k}^n m_{i,l}}
            \]
            for $1\leq i\leq k$ and
            \[
                ({C_2})_{i,i}=p^{\sum_{l=1}^k m_{l,i+k-1}}
            \]
            for $1\leq i\leq n+1-k$.
        \end{thm}
        \begin{rem}
            The signs in the right and the central factor come from the choice of the representative of the Weyl element.
        \end{rem}
        \begin{rem}
            The factor $(c_{a,b}d_{a,b}-1)$ is either $0$ if $m_{a,b}>0$, or it is $-1$ if $m_{a,b}=0$. This is an artifact of the choice to treat both cases at the same time.
        \end{rem}
        Before proceeding with the proof, the formulas are illustrated using the example Weyl element $w_0$.

        In this case $n=5$ and $k=2$. The left factor has the shape $L=\begin{psmallmatrix}
                Y&Z\\&W
            \end{psmallmatrix},$
        with $Y$ a unipotent $2\times 2$ matrix and 
        \[
            L_{1,2}=Y_{1,2}=c_{2,2}d_{1,2}p^{-m_{1,2}}+c_{2,3}d_{1,3}p^{-m_{1,2}-m_{1,3}+m_{2,2}}+c_{2,4}d_{1,4}p^{-m_{1,2}-m_{1,3}-m_{1,4}+m_{2,2}+m_{2,3}},
        \]
        where each summand corresponds to one of the three paths from $(1,2)$ to $(2,4)$. 
        For the $2\times 3$ matrix $Z$ one example entry is given by
        \[
            L_{2,3} =Z_{2,1}= d_{2,2}p^{-m_{2,2}},
        \]
        which corresponds to the only path from $(1,2)$ to $(2,2)$, and another is given by
        \begin{align*}
            L_{2,5}=Z_{2,3}=&\ c_{1,2}d_{1,4}d_{2,2}(c_{1,3}d_{1,3}-1)p^{-m_{1,3}-m_{1,4}-m_{2,4}}\\
            +&\ c_{1,3}d_{1,4}d_{2,3}(c_{2,2}d_{2,2}-1)p^{-m_{2,2}-m_{1,4}-m_{2,4}}\\
            +&\ d_{2,4}(c_{2,2}d_{2,2}-1)(c_{2,3}d_{2,3}-1)p^{-m_{22,}-m_{2,3}-m_{2,4}},
        \end{align*}
        where the summands correspond to the three different paths from $(1,4)$ to $(2,2)$.

        One example entry of the unipotent $3\times 3$ matrix $W$ is
        \[
            L_{4,5} =W_{1,2}= c_{2,3}d_{2,4}p^{-m_{2,4}}+c_{1,3}d_{1,4}p^{-m_{1,4}-m_{2,4}+m_{2,3}},
        \]
        where the summands correspond to the two paths from $(1,4)$ to $(2,3)$.

        The central factor is given by
        \[
            \begin{pmatrix}
                & & & -p^{-m_{1,2}-m_{1,3}-m_{1,4}} &\\
                & & & & -p^{-m_{2,2}-m_{2,3}-m_{2,4}}\\
                p^{m_{1,2}+m_{2,2}} & & & &\\
                & p^{m_{1,3}+m_{2,3}} & & &\\
                & & p^{m_{1,4}+m_{2,4}} & &
            \end{pmatrix}.
        \]

        For the right factor {$R=\begin{psmallmatrix}
            I_{3} & X\\ &I_2
        \end{psmallmatrix}$} look at two examples:
        \[
            R_{3,4}=X_{3,1}=c_{1,4}p^{-m_{1,4}},
        \]
        which corresponds to the only path from $(1,4)$ to $(1,4)$, and
        \begin{align*}
            R_{1,5}=X_{1,2}=\ &c_{2,2}p^{-m_{1,2}-m_{2,2}-m_{2,3}-m_{2,4}}\\+\ &c_{2,3}d_{1,3}p^{-m_{1,2}-m_{1,3}-m_{2,3}-m_{2,4}}\\+\ &c_{2,4}d_{1,4}p^{-m_{1,2}-m_{1,3}-m_{1,4}-m_{2,4}},
        \end{align*}
        where each summand corresponds to one of the three paths from $(1,2)$ to $(2,4)$.
        
        Observe that all entries next to the diagonal only contain products of at most one $c$ and one $d$. This remains true in the general case and enables the application of multiple Weyl bounds at the same time in Section \ref{sec:4}. Further away from the diagonal, products of more than one $c$ and one $d$ can appear, but this is irrelevant for the characters $\psi$ and $\psi^\prime$.

        \begin{rem}
            Theorem \ref{thm:1} and the upcoming proof can be seen as computing the Bruhat decomposition of a product of matrices. The map $b_{\underline{\alpha}}$ is defined as a product, where the Bruhat decomposition of each factor is quite simple. When figuring out what the Bruhat decomposition of the product is, all the left factors of the individual Bruhat decompositions have to be moved to the left and all the right factors to the right. Moving them conjugates them by the central factors and each other, adding $p$-powers, changing the position of the entries and potentially adding new entries. Afterwards all the (conjugated) left (resp. central or right) factors need to be multiplied. For the right factors this is simple because $U_{w_{k,n+1-k}}$ is commutative. For the left factor this is more complicated resulting in the above formulas. The entries being moved around $U$ resp. $U_w$ corresponds in spirit to moving along paths in the diagram.
            For the Weyl element $s_1\circ s_2$ for $\GL_3$ this would work as follows:
            \begin{align*}
                b_{\alpha_1}(a_{1,1})\cdot b_{\alpha_2}(a_{1,2})=&\begin{pmatrix}
                1 & d_{1,1} p^{-m_{1,1}}&\\ & 1&\\&&1
            \end{pmatrix}
            \begin{pmatrix}
                 &  -p^{-m_{1,1}}&\\ p^{m_{1,1}}& &\\&&1
            \end{pmatrix}
            \begin{pmatrix}
                1 & c_{1,1} p^{-m_{1,1}}&\\ & 1&\\&&1
            \end{pmatrix}
            \\
            &\cdot\begin{pmatrix}
                1 & &\\ & 1&d_{1,2} p^{-m_{1,2}}\\&&1
            \end{pmatrix}
            \begin{pmatrix}
                 1&  &\\ & &-p^{-m_{1,2}}\\&p^{m_{1,2}}&
            \end{pmatrix}
            \begin{pmatrix}
                1 & &\\ & 1&c_{1,1} p^{-m_{1,2}}\\&&1
            \end{pmatrix}
            \\
            =&\begin{pmatrix}
                1 & d_{1,1} p^{-m_{1,1}}&\\ & 1&\\&&1
            \end{pmatrix}
            \begin{pmatrix}
                 &  -p^{-m_{1,1}}&\\ p^{m_{1,1}}& &\\&&1
            \end{pmatrix}
            \begin{pmatrix}
                1 & &c_{1,1}d_{1,2}p^{-m_{1,1}-m_{1,2}}\\ & 1&d_{1,2} p^{-m_{1,2}}\\&&1
            \end{pmatrix}
            \\
            &\cdot\begin{pmatrix}
                1 & c_{1,1} p^{-m_{1,1}}&\\ & 1&\\&&1
            \end{pmatrix}
            \begin{pmatrix}
                 1&  &\\ & &-p^{-m_{1,2}}\\&p^{m_{1,2}}&
            \end{pmatrix}
            \begin{pmatrix}
                1 & &\\ & 1&c_{1,1} p^{-m_{1,2}}\\&&1
            \end{pmatrix}
            \\
            =&\begin{pmatrix}
                1 & d_{1,1} p^{-m_{1,1}}&\\ & 1&\\&&1
            \end{pmatrix}
            \begin{pmatrix}
                1 & &-d_{1,2} p^{-m_{1,2}-m_{1,1}}\\ & 1&c_{1,1}d_{1,2}p^{-m_{1,2}}\\&&1
            \end{pmatrix}
            \begin{pmatrix}
                 &  -p^{-m_{1,1}}&\\ p^{m_{1,1}}& &\\&&1
            \end{pmatrix}
            \\
            &\cdot\begin{pmatrix}
                 1&  &\\ & &-p^{-m_{1,2}}\\&p^{m_{1,2}}&
            \end{pmatrix}
            \begin{pmatrix}
                1 & &c_{1,1} p^{-m_{1,1}-m_{1,2}}\\ & 1&\\&&1
            \end{pmatrix}
            \begin{pmatrix}
                1 & &\\ & 1&c_{1,1} p^{-m_{1,2}}\\&&1
            \end{pmatrix}\\
            =&\begin{pmatrix}
                1 & d_{1,1} p^{-m_{1,1}}&(c_{1,1}d_{1,1}-1)\cdot d_{1,2} p^{-m_{1,2}-m_{1,1}}\\& 1&c_{1,1}d_{1,2}p^{-m_{1,2}}\\&&1
            \end{pmatrix}\\
            &\cdot\begin{pmatrix}
                 &  &-p^{-m_{1,1}-m_{1,2}}\\ p^{m_{1,1}} &&\\&p^{m_{1,2}}&
            \end{pmatrix}
            \begin{pmatrix}
                1 & &c_{1,1} p^{-m_{1,1}-m_{1,2}}\\ & 1&c_{1,1} p^{-m_{1,2}}\\&&1
            \end{pmatrix}
            \end{align*}
        \end{rem}
        
        \begin{proof}
        The proof starts with an induction on the sizes of the two blocks, namely $k$ and $n+1-k$. In each induction step the Bruhat decomposition of the new product will computed in a similar way to the above remark. Start with the case $n=k=1$, which is simply the long Weyl element for $\GL_2$ or in other words just one reflection. As seen at the beginning of the section the Bruhat decomposition in this case is
        \[
            b_\alpha(a_{1,1}) = \begin{pmatrix}
                1 & d_{1,1} p^{-m_{1,1}}\\ & 1
            \end{pmatrix}
            \begin{pmatrix}
                 &  -p^{-m_{1,1}}\\ p^{m_{1,1}}& 
            \end{pmatrix}
            \begin{pmatrix}
                1 & c_{1,1} p^{-m_{1,1}}\\ & 1
            \end{pmatrix}.
        \]
        The diagram contains only one vertex and thereby all paths also only contain this one vertex resulting in 
        $Z_{11}=d_{1,1}p^{-m_{1,1}}$ resp.\ $X_{11}=c_{1,1}p^{-m_{1,1}}$.

        \textbf{Induction on $k$:} 
        Fix $n+1-k=1$ implying $ n=k$ and $w_{k,n+1-k}=\begin{pmatrix}
        &I_k\\1&
        \end{pmatrix}$ corresponding to the diagram:

        \begin{center}
        \begin{tikzcd}
            \\
            k,k \arrow[d,dash]\\
            k-1,k \arrow[d,dash]\\
            ...  \arrow[d,dash]\\
            2,k \arrow[d,dash]\\
            1,k. \\
        \end{tikzcd}
        \end{center}

        Assume that the formulas claimed in the theorem are correct for a fixed value of $k$. To prove them for $k+1$ look at the Weyl element 
        \[
        w_{k,n+1-k}=\begin{pmatrix}
        &I_{k+1}\\1&
        \end{pmatrix}=\begin{pmatrix}
        1&&\\&&I_k\\&1&
        \end{pmatrix}\cdot s_1
        \]
        and apply the induction hypothesis to $\begin{pmatrix}
        1&&\\&&I_k\\&1&
        \end{pmatrix}$
        with renamed variables. The additional reflection $s_1$ contributes
        \[
            b_{\alpha_{1}}(a_{1,1}) = \phi_{\alpha_{1}}\left(\begin{pmatrix}
                1 & d_{1,k+1} p^{-m_{1,1}}\\ & 1
            \end{pmatrix}
            \begin{pmatrix}
                 &  -p^{-m_{1,1}}\\ p^{m_{1,1}}& 
            \end{pmatrix}
            \begin{pmatrix}
                1 & c_{1,k+1} p^{-m_{1,1}}\\ & 1
            \end{pmatrix}\right).
        \]
        Let $\fL,\fC$, and $\fR$ be the Bruhat decomposition for $w_{k+1,n+1-k}$ given by the formulas in the theorem and let
        \[
            L=\begin{pmatrix}
                1&&\\
                &Y&Z\\
                &&W
            \end{pmatrix}, \quad C=\begin{pmatrix}
                1&&\\
                &&C_1\\
                &C_2&
            \end{pmatrix},\quad \text{and } R=\begin{pmatrix}
                1&&\\
                &1&X\\
                &&I_{k-1}
            \end{pmatrix}
        \]
        be the matrices for $w_{k,n+1-k}$ with the renamed variables $x+1,k+1$ instead of $x,k$ and with an additional row on top and column in the front with a single $1$.
        Claim: $\fL\fC\fR=LCR\cdot b_{\alpha_{1}}(a_{11})$.

        \textit{Proof of claim}: Compute:
        \begin{align*}
            &LCR\cdot b_{\alpha_{1}}(a_{1,1})\\
            =& LC\begin{pmatrix}
                1 & d_{1,k+1} p^{-m_{1,k+1}}&-d_{1,k+1}p^{-m_{1,k+1}}X\\ & 1&0\\&&I_{k-1}
            \end{pmatrix}\\ \cdot&\phi_{\alpha_{1}}\left(
            \begin{psmallmatrix}
                 &  -p^{-m_{1,k+1}}\\ p^{m_{1,k+1}}& 
            \end{psmallmatrix}\right)\begin{pmatrix}
                1 & c_{1,k+1}p^{-m_{1,k+1}} &X\cdot p^{-m_{1,k+1}}\\
                &1&0\\
                &&I_{k-1}
            \end{pmatrix}.
        \end{align*}
        Check
        \[
            \fR_{1,2}=(-1)^{2}\sum_{\CP\in \CX_{1,1}}\prod_{a,b\in\CP_{rd}}c_{a,b}\prod_{a,b\in\CP_{lu}}d_{a,b}\cdot p^{-\sum_{a,b\in\CP}m_{a,b}}=c_{1,k+1}p^{-m_{1,k+1}}
        \]
        and
        \begin{align*}
            \fR_{1,j}=&(-1)^{2}\sum_{\CP\in \CX_{1,{j-1}}}\prod_{a,b\in\CP_{rd}}c_{a,b}\prod_{a,b\in\CP_{lu}}d_{a,b}\cdot p^{-\sum_{a,b\in\CP}m_{a,b}}\\
            =&(-1)^{2}\sum_{\CP\in \CX_{2,{j-1}}}\prod_{a,b\in\CP_{rd}}c_{a,b}\prod_{a,b\in\CP_{lu}}d_{a,b}\cdot p^{-\sum_{a,b\in\CP}m_{a,b}}\cdot p^{-m_{1,k+1}}\\
            =&X_{1,j-2}\cdot p^{-m_{1,k+1}},
        \end{align*}
        which proves the correctness of the right factor.
        
        The remaining factors are equal to
        \begin{align*}
            &LC\begin{pmatrix}
                1 & d_{1,k+1} p^{-m_{1,k+1}}&-d_{1,k+1}p^{-m_{1,k+1}}X\\ & 1&0\\&&I_{k-1}
            \end{pmatrix}C^{-1}\cdot C\phi_{\alpha_{1}}\left(
            \begin{pmatrix}
                 &  -p^{-m_{1,k+1}}\\ p^{m_{1,k+1}}& 
            \end{pmatrix}\right)\\
            =&\begin{pmatrix}
                1 & d_{1,k+1}p^{-m_{1,k+1}}X&d_{1,k+1}p^{-\sum_{l=1}^{k+1}m_{l,k+1}}\\&Y&Z\\&&W
            \end{pmatrix}
            \cdot\begin{pmatrix}
                &p^{-m_{1,k+1}}&\\
                &&C_1\\
                p^{m_{1,k+1}}C_2&&
            \end{pmatrix}.
        \end{align*}
        
        Check
        \[
            \fC_{1,2}=(-1)^{1}p^{-\sum_{l=k+1}^{n+1} m_{1,l}}=-p^{-m_{1,k+1}}
        \]
        and
        \[
            \fC_{j-1,j}=(-1)^{1}p^{-\sum_{l=k+1}^{n+1} m_{j-1,l}}=(C_1)_{j-2,j-2}
        \]
        for $2<j\leq k$ and also
        \[
            \fC_{k+1,1}=p^{\sum_{l=1}^{k+1} m_{l,k+1}}=p^{m_{1,k+1}}C_2,
        \]
        which proves the correctness of the central factor.
        
        Finally, check
        \begin{align*}
            \fL_{1,j}=&\sum_{\CP\in \CY_{1,j}}\prod_{a,b\in\CP_{rd}}c_{a,b}\prod_{a,b\in\CP_{lu}}d_{a,b}\cdot p^{-\sum_{a,b\in\CP}m_{a,b}+\sum_{l=k+1}^{n+1} m_{j,l}}\\
            =&d_{1,k+1}p^{-m_{1,k+1}}X_{1,j-1}
        \end{align*}
        for $2\leq j\leq k$ and 
        \begin{align*}
            &\fL_{1,k+1}\\
            =&\sum_{\CP\in \CZ_{1,1}}\prod_{a,b\in\CP_{ru}}c_{a,b}\prod_{a,b\in\CP_{ld}}d_{a,b}\prod_{a,b\in\CP_{lr}}(c_{a,b}d_{a,b}-1)\cdot p^{-\sum_{a,b\in\CP_{lr}}m_{a,b}+\sum_{a,b\in\CP_{du}}m_{a,b}-\sum_{l=1}^{k+1} m_{l,j+k}}\\
            =&d_{1,k+1}p^{-\sum_{l=1}^{k+1}m_{l,k+1}},
        \end{align*}
        which proves the correctness of the left factor because for $i>2$ the entry $\fL_{i,j}$ is given by the same formula as $L_{i,j}$. This proves the claim and concludes the induction on $k$.

        \textbf{Induction on $n+1-k$:}
        The case $n+1-k=1$ was proven for all $k$ above. Assume the formulas are correct for a Weyl element $w_{k,n+1-k}$ and turn now to $w_{k,n+2-k}$, which can be factorized as
        \[
            w_{k,n+2-k}=\begin{pmatrix}
            w_{k,n+1-k}&\\&1
            \end{pmatrix}\circ \begin{pmatrix}
            I_{n+1-k}&&\\&&I_{k}\\&1&
            \end{pmatrix} \ .
        \]
        To prove the correctness of the formulas in this case let $\fL,\fC$ and $\fR$ again be the factors of the Bruhat decomposition associated to $w_{k,n+2-k}$ as claimed in the theorem and apply the induction hypothesis to $w_{k,n+1-k}$ and the result for $n+1-k=1$ to the second factor.
        
        Let
        \[
            L=\begin{pmatrix}
                I_{n+1-k}&&\\
                &Y&Z\\
                &&1
            \end{pmatrix},\ 
            C=\begin{pmatrix}
                I_{n+1-k}&&\\
                & & C_1\\
                &C_2&
            \end{pmatrix},\text{ and }            
            R=\begin{pmatrix}
                I_{n+1-k}&&\\
                &1&X\\
                &&I_k
            \end{pmatrix},
        \]
       be the factors in the Bruhat decomposition of the product corresponding to the second factor and let 
       \[
            L^\prime=\begin{pmatrix}
                Y^\prime & Z^\prime &\\
                & W^\prime &\\
                & & 1
            \end{pmatrix},\ 
            C^\prime=\begin{pmatrix}
                & C^\prime_1&\\
                C^\prime_2 &&\\
                &&1
            \end{pmatrix}, \text{ and }
            R^\prime =\begin{pmatrix}
                I_{n+1-k} & X^\prime &\\
                & I_k &\\
                & & 1
            \end{pmatrix}
        \]
        be the factors in the Bruhat decomposition of the product corresponding to $w_{k,n+1-k}$. Claim: $\fL\fC\fR=L^\prime C^\prime R^\prime\cdot L C R$.

        \textit{Proof of claim}:
        Compute:
        \begin{align*}
            &L^\prime\cdot C^\prime\cdot R^\prime\cdot L\cdot C\cdot R\\
            =&L^\prime\cdot C^\prime \cdot \begin{pmatrix}
                I_{n+1-k} & X^\prime &\\
                & I_k &\\
                & & 1
            \end{pmatrix}\cdot \begin{pmatrix}
                I_{n+1-k}&&\\
                &Y&Z\\
                &&1
            \end{pmatrix}\cdot C\cdot R\\
            =&L^\prime\cdot C^\prime \cdot\begin{pmatrix}
                I_{n+1-k}&&X^\prime\cdot Z\\
                &Y&Z\\
                &&1
            \end{pmatrix}\cdot \begin{pmatrix}
                I_{n+1-k} & X^\prime\cdot Y &\\
                & I_k &\\
                & & 1
            \end{pmatrix}\cdot C\cdot R\\
            =&L^\prime\cdot C^\prime \cdot\begin{pmatrix}
                I_{n+1-k}&&X^\prime\cdot Z\\
                &Y&Z\\
                &&1
            \end{pmatrix}\cdot C \cdot \begin{pmatrix}
                I_{n+1-k} & & X^\prime\cdot Y\cdot C_1\cdot I_k \\
                & 1 &X\\
                & & I_k
            \end{pmatrix}\\
            =&\ L^\prime \cdot\begin{pmatrix}
                C_1^\prime Y{C_1^\prime}^{-1}&&C_1^\prime Z\\
                &I_{n+1-k}&C_2^\prime X^\prime\cdot Z\\
                &&1
            \end{pmatrix}\cdot C^\prime\cdot C\cdot \begin{pmatrix}
                I_{n+1-k} & & X^\prime\cdot Y\cdot C_1\cdot I_k \\
                & 1 &X\\
                & & I_k
            \end{pmatrix}\\
            =&\begin{pmatrix}
                Y^\prime\cdot C_1^\prime Y{C_1^\prime}^{-1}&Z^\prime&Y^\prime C_1^\prime Z +Z^\prime C_2^\prime X^\prime\cdot Z\\
                & W^\prime & W^\prime C_2^\prime X^\prime\cdot Z\\
                &&1
            \end{pmatrix}\cdot C^\prime\cdot C\cdot \begin{pmatrix}
                I_{n+1-k} & & X^\prime\cdot Y\cdot C_1\cdot I_k \\
                & 1 &X\\
                & & I_k
            \end{pmatrix}\\.
        \end{align*}

        \textbf{The central factor:} Check
        \begin{align*}
            \fC_{i,n+2-k+i}=&(-1)^{n+2-k}p^{-\sum_{l=k}^{n+1} m_{i,l}}\\
            =&(-1)^{n+1-k}p^{-\sum_{l=k}^n m_{i,l}}\cdot (-1)^{k+1-k}p^{-\sum_{l=n+1}^{n+1} m_{i,l}}\\
            =&C^\prime_{i,n+1-k+i} C_{n+1-k+i,n+2-k+i}
        \end{align*}
        for $1\leq i\leq k$ and
        \begin{align*}
            \fC_{i,i-k}=p^{\sum_{l=1}^k m_{l,i-1}}\cdot 1=C^\prime_{i,i-k}\cdot C_{i-k,i-k}
        \end{align*}
        for $k<i\leq n+1$ and finally
        \[
            \fC_{n+2,n+2-k}=1\cdot p^{\sum_{l=1}^k m_{l,n+1}}=C^\prime_{n+2,n+2}\cdot C_{n+2,n+2-k},
        \]
        and
        \[
            \fC_{i,j}=0
        \]
        else, which proves the correctness of the central factor.

        Before checking the left and the right factors recall the diagram used to define the formulas

        \begin{center}
          
        \begin{tikzcd}
            & & & &\\
            k,k \arrow[r,dash] \arrow[d,dash]&  ... \arrow[r,dash] & k,n-1 \arrow[r,dash] \arrow[d,dash] &k,n\arrow[d,dash]\arrow[r,dash] & \textbf{k,n+1}  \arrow[d,dash] \\
            k-1,k \arrow[r,dash] \arrow[d,dash]& ... \arrow[r,dash] & k-1,n-1 \arrow[r,dash] \arrow[d,dash] &k-1,n\arrow[d,dash]\arrow[r,dash]& \textbf{k-1,n+1}  \arrow[d,dash] \\
            ...  \arrow[d,dash]& &...  \arrow[d,dash] & ...  \arrow[d,dash] &\textbf{...}\arrow[d,dash]\\
            2,k \arrow[r,dash] \arrow[d,dash] & ... \arrow[r,dash] & 2,n-1 \arrow[r,dash] \arrow[d,dash] &2,n\arrow[d,dash]\arrow[r,dash]& \textbf{2,n+1}  \arrow[d,dash]\\
            1,k \arrow[r,dash] & ... \arrow[r,dash] & 1,n-1 \arrow[r,dash] &1,n\arrow[r,dash]& \textbf{1,n+1},\\
            & & & &
        \end{tikzcd}  
        \end{center}
        
        where the bold column is the part corresponding to the added reflections \[s_{n+1}\circ...\circ s_{n+2-k}=\begin{pmatrix}
            I_{n+1-k}&&\\&&I_{k}\\&1&
            \end{pmatrix}\ .\]
        To avoid confusion between the formulas for $\fL,\fC,\fR$ and $L^\prime,C^\prime,R^\prime$, let $\CX,\CY,\CZ$, and $\CW$ refer to sets of paths in this diagram and $\CX^\prime,\CY^\prime,\CZ^\prime$, and $\CW^\prime$ refer to paths in the diagram without the bold vertices.
        \textbf{The right factor:} Recall the right factor
        \[
            \begin{pmatrix}
                I_{n+1-k} & & X^\prime\cdot Y\cdot C_1 \\
                & 1 &X\\
                & & I_k
            \end{pmatrix}.
        \]
        Before checking the formulas some calculations are needed.
        For $l< j-(n+2-k)$ compute
        \begin{align*}
            X^\prime_{i,l}Y_{l,j-(n+2-k)}=&(-1)^{i+n+1-k}(\sum_{\CP\in \CX_{i,l}^\prime}\prod_{a,b\in\CP_{rd}}c_{a,b}\prod_{a,b\in\CP_{lu}}d_{a,b}\cdot p^{-\sum_{a,b\in\CP}m_{a,b}})\\
            &\cdot c_{j-(n+2-k),n+1}d_{l,n+1}p^{-\sum_{h=l}^{j-(n+2-k)-1}m_{h,n+1}}
        \end{align*}
        and also
        \[
            X^\prime_{i,j-(n+2-k)}Y_{j-(n+2-k),j-(n+2-k)}=(-1)^{i+n+1-k}(\sum_{\CP\in \CX_{i,j-(n+2-k)}^\prime}\prod_{a,b\in\CP_{rd}}c_{a,b}\prod_{a,b\in\CP_{lu}}d_{a,b}\cdot p^{-\sum_{a,b\in\CP}m_{a,b}}).
        \]
        In both cases the $Y$-factor extends the path in $\CX_{i,l}^\prime$ going from $(1,k+i-1)$ to ${(l,n)}$ to the vertex ${(j-(n+2-k),n+1)}$ and multiplying with 
        \[
            ({C_1})_{j-(n+2-k),j-(n+2-k)}=-p^{-m_{j-(n+2-k),n+1}}
        \]
        adds the $p$-power and factor $-1$ that are missing for a summand in the formula of $\fR$. Each path in $\CX_{i,j-(n+2-k)}$ contains an edge from $(l,n)$ to $(l,n+1)$ for some $l$. Splitting the path at this edge gives a path in $\CX_{i,l}^\prime$ and a path only going up from $(l,n+1)$ to $(j-(n+2-k),n+1)$. Therefore, each path in $\CX_{i,j-(n+2-k)}$ is obtained as a summand in such a product $X^\prime_{i,j-(n+2-k)}Y_{j-(n+2-k),j-(n+2-k)}$.
        Thus, check
        \begin{align*}
            \fR_{i,j}=&(-1)^{i+n+2-k}\sum_{\CP\in \CX_{i,j-(n+2-k)}}\prod_{a,b\in\CP_{rd}}c_{a,b}\prod_{a,b\in\CP_{lu}}d_{a,b}\cdot p^{-\sum_{a,b\in\CP}m_{a,b}}\\
            =&\sum_{l=1}^{k}X^\prime_{i,l}Y_{l,j-(n+2-k)}\cdot{C_1}_{j-(n+2-k),j-(n+2-k)}
        \end{align*}
        for $1\leq i<n+2-k<j\leq n+2$. 
        For $i=n+2-k<j\leq n+2$ the formulas directly give
        \[
            \fR_{ij}=X_{1,j-(n+2-k)},
        \]
        which proves the correctness of the right factor.

        \textbf{The left factor:} Recall the left factor:
        \begin{align*}
            \begin{pmatrix}
                Y^\prime\cdot C_1^\prime Y{C_1^\prime}^{-1}&Z^\prime&Y^\prime C_1^\prime Z +Z^\prime C_2^\prime X^\prime\cdot Z\\
                & W^\prime & W^\prime C_2^\prime X^\prime\cdot Z\\
                &&1
            \end{pmatrix}.
        \end{align*}

        For $l<j\leq k$ compute
        \begin{align*}
            Y^\prime_{i,l}Y_{l,j}&=\sum_{\CP\in \CY_{i,l}^\prime}\prod_{a,b\in\CP_{rd}}c_{a,b}\prod_{a,b\in\CP_{lu}}d_{a,b}\cdot p^{-\sum_{a,b\in\CP}m_{a,b}+\sum_{h=k}^n m_{lh}}\cdot c_{j,n+1}d_{l,n+1}p^{-\sum_{h=l}^{j}m_{h,n+1}+m_{j,n+1}}
        \end{align*}
        and for $l=j$ compute
        \begin{align*}
            Y^\prime_{i,j}Y_{j,j}&=\sum_{\CP\in \CY_{i,l}^\prime}\prod_{a,b\in\CP_{rd}}c_{a,b}\prod_{a,b\in\CP_{lu}}d_{a,b}\cdot p^{-\sum_{a,b\in\CP}m_{a,b}+\sum_{h=k}^n m_{l,h}}.
        \end{align*}
        In both cases the $Y$-factor extends the path in $\CY_{i,l}^\prime$ going from $(i,k)$ to $(l,n)$ to the vertex $(j,n+1)$, multiplying this with ${(C_1^\prime)}_{l,l}$ cancels the factor $p^{\sum_{h=k}^n m_{l,h}}$, and multiplying with ${(C_1^\prime)}_{j,j}^{-1}$ adds the missing $p$-power to achieve the following equality for $1\leq i<j\leq k$:
        \begin{align*}
            \fL_{i,j}&=\sum_{\CP\in \CY_{i,j}}\prod_{a,b\in\CP_{rd}}c_{a,b}\prod_{a,b\in\CP_{lu}}d_{a,b}\cdot p^{-\sum_{a,b\in\CP}m_{a,b}+\sum_{l=k}^{n+1} m_{j,l}}\\
            &=\sum_{l=1}^k Y^\prime_{i,l}{(C_1^\prime)}_{l,l}Y_{l,j}{(C_1^\prime)}_{j,j}^{-1}.
        \end{align*}
        As for the right factor, each path in $\CY_{i,j}$ can be split into a path in $\CY_{i,l}^\prime$ and a path only going up from $(l,n+1)$ to $(j,n+1)$. Therefore, each path is obtained as a summand in such a product $Y^\prime_{i,l}{(C_1^\prime)}_{l,l}Y_{l,j}{(C_1^\prime)}_{j,j}^{-1}$.
        
        For $1\leq i\leq k<j\leq n+1$ by definition
        \[
            \fL_{i,j}=Z^\prime_{i,j-k}
        \]
        holds and for $k<i\leq j\leq n+1$ by definition
        \[
            \fL_{i,j}=W^\prime_{i-k,j-k}
        \]
        holds and the only thing left to check is the last column.

        For $i>k$ the equality
        \[
            \fL_{i,n+2}=\sum_{l=1}^{k}\sum_{j=i-k}^{n+1-k} W^\prime_{i-k,j}{(C_2^\prime)}_{j,j}X^\prime_{j,l}Z_{l,1}
        \]
        needs to be proven. The idea is that many of the summands cancel each other and the remaining can be grouped to correspond to the paths in the formula for $\fL_{i,n+2}$.
        
        \textit{The following steps are demonstrated using the example of an induction step from $w_{3,3}$ to $w_{3,4}$. Recall the diagram}
        \begin{center}
        \begin{tikzcd}
            3,3 \arrow[r,dash] \arrow[d,dash]& 3,4 \arrow[r,dash] \arrow[d,dash] & 3,5 \arrow[r,dash] \arrow[d,dash] & \textbf{3,6} \arrow[d,dash] \\
            2,3 \arrow[r,dash] \arrow[d,dash]& 2,4 \arrow[r,dash] \arrow[d,dash]& 2,5 \arrow[r,dash] \arrow[d,dash] & \textbf{2,6} \arrow[d,dash] \\
            1,3 \arrow[r,dash] & 1,4 \arrow[r,dash] & 1,5 \arrow[r,dash] & \textbf{1,6} . 
        \end{tikzcd}
        \end{center}
        
        Fix $1\leq l\leq k$ and $i>k$. The factor $W^\prime_{i-k,j}$ contains a sum over paths from $(1,j+k-1)$ to $(k,i-1)$. For such a path $\CP_W$ let $h(\CP_W)$ be such that $h(\CP_W),j+k-2$ is the first vertex in $\CP_W$ in the $(*,j+k-2)$ column or $h(\CP_W)=k$ if $i-1=j+k-1$.
        The factor $X^\prime_{j,l}$ contains a sum over paths from $(1,j+k-1)$ to $(l,n)$. For such a path $\CP_X$ let $h(\CP_X)$ be such that $h(\CP_X),j+k$ is the first vertex in $\CP_X$ in the $(*,j+k)$ column or $h(\CP_X)=l$ if $j+k-1=n$.
        For a pair of such paths $(\CP_W,\CP_X)$ define 
        \[
            h(\CP_W,\CP_X):=\min(h(\CP_W),h(\CP_X))
        \]
        and let $s(\CP_W,\CP_X)$ be such that $(h(\CP_W,\CP_X),s(\CP_W,\CP_X))$ is the last vertex in $\CP_W$ in the $(h(\CP_W,\CP_X),*)$ row and also let $t(\CP_W,\CP_X)$ be such that $(h(\CP_W,\CP_X),t(\CP_W,\CP_X))$ is the last vertex in $\CP_X$ in the $(h(\CP_W,\CP_X),*)$ row.
        
        \textit{In the example with $n=5$, $k=3$, fix $l=2$ and $i=4$. Now, for $j=2$ there are three paths $w_1,w_2,w_3$ from $(1,4)$ to $(3,3)$ contributing to $W^\prime$ and two paths $x_1,x_2$ from $(1,4)$ to $(2,5)$ contributing to $X^\prime$.}
        \usetikzlibrary{decorations.pathmorphing}
        \begin{center}
        \begin{tikzcd}
            \textbf{3,3} \arrow[r,dashleftarrow,"w_1"] \arrow[d,leftsquigarrow,shift right]\arrow[d,leftarrow,shift left]& 3,4\arrow[d,dashleftarrow] & 3,5  & 3,6 &&3,3 & 3,4  & 3,5  & {3,6}\\
            2,3 \arrow[r,leftarrow,"w_2"] \arrow[d,leftsquigarrow]& 2,4 \arrow[d,dashleftarrow, shift left]\arrow[d,leftarrow, shift right]& 2,5 & {2,6}&& 2,3 & 2,4 \arrow[r,dashrightarrow,"x_1"] \arrow[d,dashleftarrow]& \textbf{2,5} \arrow[d,leftarrow] & {2,6} \\
            1,3 \arrow[r,leftsquigarrow,"w_3"] & \textbf{1,4}  & 1,5& {1,6}&&1,3  & \textbf{1,4} \arrow[r,rightarrow, "x_2"] & 1,5  & {1,6} . 
        \end{tikzcd}
        \end{center}
        
        \textit{For the pair of paths $(w_1,x_1)$ the values would be $h(w_1)=3$, $h(x_1)=2$, $h(w_1,x_1)=2$, $s(w_1,x_1)=4$, and $t(w_1,x_1)=5$ and for the pair of paths $(w_3,x_2)$ the values would be $h(w_3)=1$, $h(x_2)=1$, $h(w_3,x_2)=1$, $s(w_3,x_2)=3$, and $t(w_3,x_2)=5$.}
        
        Now, compute 
        \[
            \sum_{j=i-k}^{n+1-k} W^\prime_{i-k,j}{(C_2^\prime)}_{j,j}X^\prime_{j,l}Z_{l,1}
        \]
        by grouping the pairs of paths first by $h(\CP_W,\CP_X)$ and then by $s(\CP_W,\CP_X)$ and $t(\CP_W,\CP_X)$. Start with the paths with $h(\CP_W,\CP_X)=l$, which implies $t(\CP_W,\CP_X)=n$. For a fixed value of $s=s(\CP_W,\CP_X)$ consider only the parts of the paths between $(1,j+k-1)$ and $(h(\CP_W,\CP_X),s)$ resp.\ between $(1,j+k-1)$ and $(h(\CP_W,\CP_X),t)$. Their contribution is
        \begin{align*}
            &\sum_{j=s+2-k}^{n+1-k}  c_{l,s}d_{l,j+k-1}\prod_{q=s+1}^{j+k-2} (c_{q,l}d_{q,l}-1)p^{-\sum_{q=s+1}^{j+k-2} m_{q,l}+\sum_{q=1}^{l-1} m_{q,j+k-1}-\sum_{q=1}^k m_{q,j+k-1}}\\
            \cdot& (-1)^{j+n+1-k} c_{l,j+k-1}p^{-\sum_{q=1}^{l-1} m_{q,j+k-1}-\sum_{q=j+k-1}^n m_{l,q}}\\
            +& (-1)^{s+1-k+n+1-k} c_{l,s}p^{-\sum_{q=1}^{l-1} m_{q,s}-\sum_{q=s}^n m_{l,q}},
        \end{align*}
        where the last summand corresponds to $j=s+1-k$. After multiplication with $(C_2^\prime)_{j,j}$, this can be rewritten as
        \[
            c_{l,s}\prod_{q=s+1}^{n} (c_{q,l}d_{q,l}-1)p^{-\sum_{q=s+1}^{n} m_{q,l}}\ ,
        \]
        where the summand for a fixed $j$ in the first formula is obtained by multiplying out the factors with $q\geq j+k$ in the second formula and picking the $-1$ for all $q>j+k$ and $c_{q,l}d_{q,l}$ for $q=j+k$.
        
        Together with the remaining part of the $W^\prime$-path and the factor $Z_{l,1}$ this results in the formula for $\fL_{i,n+2}$ for the paths containing $(l,n+1)$ and $(l,s)$ but not $(l,s-1)$. Summing over the possibilities for $s$ and including the $l$-sum gives the desired formula.
        
        \textit{In the example fix $s=4$. There are two pairs of paths with $h=2$ and $s=4$, the pair $(w_1,x_1)$ for $j=2$ and $(w_4,x_3)$ for $j=3$. Their contributions together with the additional factors give the contribution of the path $z_1$.}
        
        
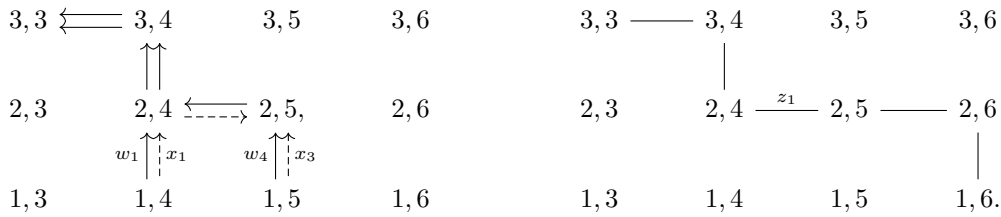
\begin{center}
        \begin{tikzcd}
            {3,3} \arrow[r,leftarrow,shift right]\arrow[r,leftarrow,shift left] & 3,4\arrow[d,leftarrow, shift left]\arrow[d,leftarrow, shift right] & 3,5  & 3,6 &&3,3 \arrow[r,dash]& 3,4 \arrow[d,dash] & 3,5  & {3,6}\\
            2,3 & 2,4 \arrow[d,leftarrow, shift right,"w_1"']\arrow[d,dashleftarrow, shift left,"x_1"]\arrow[r,leftarrow, shift left]\arrow[r,dashrightarrow, shift right]& 2,5 \arrow[d,dashleftarrow,"x_3", shift left],\arrow[d,leftarrow,"w_4"', shift right]& {2,6}&& 2,3 & 2,4 \arrow[r,dash,"z_1"] & {2,5} \arrow[r,dash]& {2,6}\arrow[d,dash] \\
            1,3 & {1,4}  & 1,5& {1,6}&&1,3  & {1,4}& 1,5  & {1,6} . 
        \end{tikzcd}
        \end{center}

        Claim: All other summands in $\sum_{j=i-k}^{n+1-k} W^\prime_{i-k,j}{(C_2^\prime)}_{j,j}X^\prime_{j,l}Z_{l,1}$ cancel each other.

        \textit{Proof of claim:} With the same grouping the sum barely changes. The only difference is that the vertex $(h(\CP_W,\CP_X),t)$ has an outgoing edge upwards instead of to the right in the $X^\prime$-path. Fix $h=h(\CP_W,\CP_X)<l$ and $s=s(\CP_W,\CP_X)$ and $t=t(\CP_W,\CP_X)$.
        The contribution of the paths with these parameters between $(1,j+k-1)$ and $(h(\CP_W,\CP_X),s)$ resp.\ between $(1,j+k-1)$ and $(h(\CP_W,\CP_X),t)$ together with the factor $(C_2^\prime)_{j,j}$ is
        \begin{align*}
            &\sum_{j=s+2-k}^{t+1-k}  c_{h,s}d_{h,j+k-1}\prod_{q=s+1}^{j+k-2} (c_{q,h}d_{q,h}-1)p^{-\sum_{q=s+1}^{j+k-2} m_{q,h}+\sum_{q=1}^{h-1} m_{q,j+k-1}-\sum_{q=1}^k m_{q,j+k-1}}\\
            \cdot& (-1)^{j+n+1-k} d_{h,t}c_{h,j+k-1}p^{-\sum_{q=1}^{h-1} m_{q,j+k-1}-\sum_{q=j+k-1}^t m_{h,q}}(C_2^\prime)_{j,j}\\
            +&c_{h,s}d_{h,t}\prod_{q=s+1}^{t-1} (c_{q,h}d_{q,h}-1)p^{-\sum_{q=s+1}^{t-1} m_{q,h}+\sum_{q=1}^{h-1} m_{q,t}-\sum_{q=1}^k m_{q,t}}\\
            \cdot& (-1)^{n+t} p^{-\sum_{q=1}^{h-1} m_{q,t}- m_{h,t}}(C_2^\prime)_{t,t}\\
            +& (-1)^{s+1-k+n+1-k} d_{h,t}c_{l,s}p^{-\sum_{q=1}^{h-1} m_{q,s}-\sum_{q=s}^t m_{h,q}}(C_2^\prime)_{s+1-k,s+1-k},
        \end{align*}
        where a factor $d_{h,t}$ was added for all summands with $j+1-k<t$
        and the factor $c_{h,t}$ was removed for $j+1-k=t$. If $m_{h,t}=0$, this sum is zero because $d_{h,t}=0$ divides the sum. If $m_{h,t}>0$, the sum is zero because $(c_{t,h}d_{t,h}-1)=0$ divides the sum. Thus, the sums over paths with fixed $h,t$, and $s$ cancel each other, and the total contribution of pairs of paths with $h<l$ is zero, which proves the correctness of the formula for $\fL_{i,n+2}$ for $i>k$.
        
        \begin{rem}
            In the above computation, the corner case $i=j$ was not treated separately because $ W_{i,i}=1$.
        \end{rem}
        \textit{In the example fix $h=1$, $s=3$ and $t=4$. There are two pairs of paths which satisfy this, the pair $(w_3,x_1)$ for $j=2$ and the pair $(w_5,x_4)$ for $j=1$. Their contributions to the sum are the same except for the sign, which results in them cancelling. This is caused by the vertex $(1,4)$ having an outgoing edge upwards for both the $x$-paths.}
        
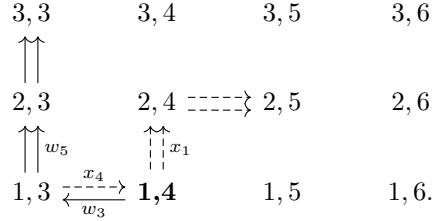
\begin{center}
        \begin{tikzcd}
            {3,3} \arrow[d,leftarrow,shift right]\arrow[d,leftarrow,shift left]& 3,4 & 3,5  & 3,6 \\
            2,3 \arrow[d,leftarrow,shift right]\arrow[d,leftarrow,shift left,"w_5"]& 2,4 \arrow[d,dashleftarrow, shift left, "x_1"]\arrow[d,dashleftarrow, shift right]\arrow[r,dashrightarrow, shift left]\arrow[r,dashrightarrow, shift right] & 2,5 & {2,6}\\
            1,3 \arrow[r,leftarrow,"w_3"', shift right]\arrow[r, dashrightarrow, shift left,"x_4"] & \textbf{1,4}  & 1,5& {1,6} . 
        \end{tikzcd}
        \end{center}
        
        Claim: For $i<k$:
        \begin{equation}
            \label{eq:1}
            \fL_{i,n+2}=\sum_{l=i}^k Y^\prime_{i,l}({C_1^\prime})_{l,l}Z_{l,1}+\sum_{l=1}^{k}\sum_{j=1}^{n+1-k} Z^\prime_{i,j}({C_2^\prime})_{j,j}X^\prime_{j,l}Z_{l,1}.
        \end{equation}
        \textit{Proof of claim:} Look at
        \[
            \sum_{l=1}^{i-1}\sum_{j=1}^{n+1-k} Z^\prime_{i,j}({C_2^\prime})_{j,j}X^\prime_{j,l}Z_{l,1}.
        \]
        This can be computed exactly as above with $W^\prime$ replaced by $Z^\prime$ because the only difference, in this case, is the treatment of the last vertex, which cannot be $h,s$ because $h\leq l<i$. Thus, this part of the sum results in the desired formula restricted to the paths containing $(l,n+1)$ but not $(l+1,n+1)$.

        For the summand $l=i$ the argument for cancellation can again be repeated. The contribution of the paths that do not cancel changes because the vertex $i,k$ is treated with an outgoing edge to the left. This results in
        \begin{align*}
            &\sum_{j=k}^n  d_{i,j}\prod_{q=k}^{j-1} (c_{q,i}d_{q,i}-1)p^{-\sum_{q=k}^{j-1} m_{q,i}+\sum_{q=1}^{i-1} m_{q,j}-\sum_{q=1}^k m_{q,j}}\cdot (-1)^{j+n} c_{i,j}p^{-\sum_{q=1}^{i-1} m_{q,j}-\sum_{q=j}^n m_{i,q}}({C_2^\prime})_{j,j}Z_{l,1}.
        \end{align*}
        Adding $Y^\prime_{i,i}{(C_1^\prime)}_{i,i}Z_{i,1}$, which corresponds to the case $l=i$ in the first sum of (\ref{eq:1}), results in
        \[
            \prod_{q=k}^{n} (c_{q,i}d_{q,i}-1)d_{i,n+1}p^{-\sum_{q=k}^n m_{q,i}+\sum_{q=1}^{i-1}m_{q,n+1}-\sum_{q=1}^k m_{q,n+1}},
        \]
        which is the contribution of the path containing $(1,n+1)$ and $(i,n+1)$ and $(i,k)$.

       Fix $l>i$. In the second sum in (\ref{eq:1}) the pairs of paths with $h<i$ cancel each other for the same reasons as before. The remaining pairs with $h=i$ do not cancel but together with the term $Y^\prime_{i,l}({C_1^\prime})_{l,l}Z_{l,1}$ from the first sum in (\ref{eq:1}) they do. This concludes the proof of the correctness of the left factor and thus of Theorem \ref{thm:1}.
        \end{proof}
        
        \begin{thm}
        \label{thm:2}
            Let $w=w_{k,n+1-k}$ be as in Theorem \ref{thm:1}. As in \cite{DR98}*{Proposition 3.1} the map $b_{\underline\alpha}$ is injective from $Y(\underline m)$ to $U(\IZ_p)\backslash C(\fn)$. Define
            \[
                M_{i,j}:={\sum_{l=1}^{i-1}m_{l,j}+\sum_{l=j+1}^{n+1}m_{i,l}+m_{i,j}}.
            \]
            The set
            \[
                C_w(\underline m):=\{\underline c = (c_{i,j})_{1\leq i\leq k\leq j\leq n}\in \IZ^{k\cdot (n+1-k)}:0\leq c_{i,j}< C_{i,j}:=p^{M_{i,j}}\text{ with } (c_{i,j},p^{m_{i,j}})=1\text{ for all } i,j\}
            \]
            parametrizes a system of representatives not equivalent under the right action by $U_{w}(\IZ_p)$ on $Y(\underline m)$. The sizes of $Y(\underline m)/U_{w}(\IZ_p)$ and $X(\fn)$ computed in Lemma \ref{lem:2} imply that
            \[
                \bigcup_{\underline m\in \CM_{w}(r)} C(\underline m)
            \]
            parametrizes a full system of representatives for the right action by $U_{w}(\IZ_p)$ on $U(\IZ_p)\backslash C(\fn)$.
        \end{thm}

        \begin{proof} It only needs to be shown that the given representatives are not equivalent because the total number is equal to the expected number computed in Lemma \ref{lem:2}. Only the right factor of the Bruhat decomposition is relevant because the left action of $U(\IZ_p)$ is already accounted for. Let $R$ (resp.\ $R^{\prime}$) be the right factor for $\underline c$ (resp.\ $\underline c^\prime$) both in $C(\underline m)$. Let $A\in U_w(\IZ_p)$ be such that $RA=R^\prime$. By definition $R,R^\prime$ and $A$ are in $U_w(\IQ_p)$ and thus of the shape
        \[
            \begin{pmatrix}
                I_{n+1-k} & *\\
                0 & I_k
            \end{pmatrix}.
        \]
        Therefore, $RA=R^\prime$ is equivalent to
        \begin{equation}
            \label{eq:7}
            R_{i,j}+A_{i,j}=R^\prime_{i,j}
        \end{equation}
        for $1\leq i\leq n+1-k< j\leq n$. For $i=n+k-1$ and $j=n+2-k$ this is
        \[
        	c_{1,n}\cdot p^{-m_{1,n}}+A_{n+1-k,n+2-k}=c^\prime_{1,n}\cdot p^{-m_{1,n}}.
        \]
        Because $A_{n+1-k,n+2-k}\in\IZ_p$, this implies 
        \[
        	c_{1,n}\equiv c^\prime_{1,n}\mod p^{-m_{1,n}}
        \]
        and by definition of the $c_{i,j}$
        \[
        	c_{1,n}= c^\prime_{1,n}
        \]
        and $A_{n+1-k,n+2-k}=0$.
        
        Define $i^\prime:=j+k-(n+1)$ and $j^\prime:=k+i-1$ and define an ordering on the vertices of the diagram as follows $a,b<c,d$ if $a<c$ or if $a=c$ and $b>d$. 
        For each entry $R_{i,j}$ the path from $1,k+i-1$ going up to $j+k-(n+1),k+i-1$ and then right to $j+k-(n+1),n$ contributes the summand
        \[
            c_{i^\prime, j^\prime}\cdot p^{-\sum_{l=1}^{i^\prime-1}m_{l,j^\prime}-\sum_{l=j^\prime+1}m_{i^\prime, l}-m_{i^\prime, j^\prime}}=c_{i^\prime, j^\prime}\cdot p^{-M_{i^\prime, j^\prime}}.
        \]
        All other paths contributing to $R_{i,j}$ contain only variables $c_{a,b},d_{a,b}$ and $p^{\pm m_{a,b}}$ with $a,b<i^\prime, j^\prime$.
        
        Now, $c_{i^\prime, j^\prime}=c_{i^\prime, j^\prime}^\prime$ is proven by induction on the defined order, where the smallest vertex $1,n$ was treated above.
        By induction all summands besides $c_{i^\prime, j^\prime}\cdot p^{-M_{i^\prime, j^\prime}}$ are already equal for $R$ and $R^\prime$. Thus, (\ref{eq:7}) is equivalent to
        \[
            c_{i^\prime, j^\prime}\cdot p^{-M_{i^\prime, j^\prime}}+A_{i,j}=c^\prime_{i^\prime, j^\prime}\cdot p^{-M_{i^\prime, j^\prime}}.
        \]
        Because $A_{i,j}\in\IZ_p$, this implies
        \[
            c_{i^\prime, j^\prime}\equiv c^\prime_{i^\prime, j^\prime} \mod p^{M_{i^\prime, j^\prime}}
        \]
        and by definition of the $c_{i,j}$
        \[
            c_{i^\prime, j^\prime}=c^\prime_{i^\prime, j^\prime}
        \]
        and $A_{i,j}=0$. This concludes the proof.
        \end{proof}
	Combining Theorem \ref{thm:1} and Theorem \ref{thm:2} gives Theorem \ref{thm:4} for simple Weyl elements as a corollary.
	\begin{cor}
         \label{cor:1}
             For $\fn=\fc\cdot w_{k,n+1-k}\in\GL_{n+1}$ with exponent vector $ r$ the Kloosterman sum can be parametrized by
             \[
                 \Kl_p(\psi,\psi^\prime,\fn)=\sum_{\underline m\in\CM_{w_{k,n+1-k}}}(r)\Kl_p(\underline m,\psi,\psi^\prime,w_{k,n+1-k}),
             \]
             with $\Kl_p(\underline m,\psi,\psi^\prime,w_{k,n+1-k})$ defined by
             \begin{align*}
                 \sum_{\underline c\in C_{w_{k,n+1-k}}}(\underline m)\Psi\big(&\sum_{i=1}^{k-1}\psi_i(\sum_{j=k}^{n} c_{i+1,j}d_{i,j}p^{-m_{i,j}+\sum_{l=k}^{j-1}m_{i+1,l}-m_{i,l}})+\psi_kc_{k,k}p^{-m_{k,k}}\\
                 +&\sum_{i=k+1}^{n}\psi_i(\sum_{j=1}^{k} c_{j,i-1}d_{j,i}p^{-m_{j,i}+\sum_{l=j+1}^{k}m_{l,i-1}-m_{l,i}})+\psi^\prime_{n+1-k}c_{1,n}p^{-m_{1,n}}\big).
             \end{align*}
         \end{cor}
	
    \subsection{General Weyl elements}
        The main idea to treat general Weyl elements is an induction on the number of blocks. An additional block is added by partitioning the top right block into two blocks. This adds new variables $c_{i,j}$ that do not depend on the already existing $m_{i,j}$ and the old variables do not depend on the new ones. Here $c_{i,j}$ depending on $m_{i^\prime, j^\prime}$ means that $m_{i^\prime, j^\prime}$ appears in $C_{i,j}$.

        The example Weyl element in this section is $w_1$ defined as follows
        \[
            w_1:=\begin{pmatrix}
                & & & & 1\\
                & & & 1 &\\
                1 & & & &\\
                & 1 & & &\\
                & & 1 & &\\
            \end{pmatrix}=w_0\circ s_4.
        \]
        This Weyl element has $3$ blocks and is written as a Weyl element with two blocks times the permutation which separates the top block into two (in this case just one reflection). This second factor can be interpreted as a Weyl element with two blocks for $\GL_k$, where $k$ is the size of the big block that is separated.

        In general, let a Weyl element $w$ for $\GL_{N+1}$ have more than $2$ blocks, let $n+1$ be the combined size of the top two blocks, and let $k$ be the size of the top block. Write
        \[
            w=w^\prime\circ \begin{pmatrix}
                I_{N+1-(n+1)} &  & \\
                & & I_k\\
                & I_{n+1-k}&
            \end{pmatrix},
        \]
        where $w^\prime$ has the same matrix representation as $w$ except for the top two blocks being replaced by one block of size $n+1$. Therefore, $w^\prime$ has one block less than $w$ and the induction hypothesis can be applied to it. The second factor is itself not an admissible  Weyl element but the bottom right $(n+1)\times(n+1)$ matrix has the shape of an admissible Weyl element for $\GL_{n+1}$ and the results from the previous section can be used to compute the Bruhat decomposition of the corresponding factors of $b_{\underline \alpha}$. Combining both Bruhat decompositions enables the computation of the entries of the first off-diagonal of $L$ and $R$ exactly and the computation of $R$ up to the right action of $U_w(\IZ_p)$. 

        \begin{thm}
        \label{thm:3}
            Let $w$ be a Weyl element for $\GL_{N+1}$ with more than $2$ blocks, let $n+1$ be the combined size of the top two blocks, let $k$ be the size of the top block, and let $f$ be the size of the third block from the top. As above write
            \[
                w=w^\prime\circ \begin{pmatrix}
                    I_{N+1-(n+1)} &  & \\
                    & & I_k\\
                    & I_{n+1-k}&
                \end{pmatrix}.
            \]
            Let $L^\prime\cdot C^\prime\cdot R^\prime$ be the Bruhat decomposition of $b_{\underline\alpha^\prime}$ corresponding to $w^\prime$. Let $L\cdot C\cdot R$ be the Bruhat decomposition corresponding to the Weyl element
            \[
                w_{k,n+1-k}:=\begin{pmatrix}
                    & I_k\\
                    I_{n+1-k} &
                \end{pmatrix}.
            \]
            And let $\fL\cdot\fC\cdot \fR$ be the Bruhat decomposition of $b_{\underline\alpha}(\underline a)$ corresponding to $w$.

            The left factor satisfies:
            \[
                \fL_{i,i+1}=L^\prime_{i,i+1}+L_{i,i+1}\cdot p^{\sum_{j=n+1}^{N-1}m_{i,j}-m_{i+1,j}}
            \]
            for $1\leq i\leq n$ and
            \[
                \fL_{i,i+1}=L^\prime_{i,i+1}
            \]
            for $n+1\leq i\leq N$.
            
            The central factor satisfies:
            \[
                \fC=C^\prime\cdot \begin{pmatrix}
                I_{N+1-(n+1)} &\\
                & C
            \end{pmatrix}.
            \]
            
            The right factor satisfies:
            \[
                \fR_{i,i+1}=R^\prime_{i,i+1}
            \]
            for $1\leq i< N+1-(n+1)$,
            \[
                \fR_{i,i+1}=\sum_{j=1}^{k+1}c_{j,n+f}d_{j,k}\cdot p^{\sum_{l=1}^{j-1}m_{l,k}-m_{l,n+f}-m_{j,n+f}}
            \]
            for $i=N+1-(n+1)$ and
            \[
                \fR_{i,j}=R_{d,e}
            \]
            for $i=N+1-(n+1)+d$ and $j=N+1-(n+1)+e$ with $1\leq d,e\leq n$.
        Furthermore, a system of representatives for the right action of $U_w(\IZ_p)$ can be parametrized as follows. Partition {$\underline m=(\underline m_1,\underline m_2)\in \CM_w(r)$} into two parts belonging to $w^\prime$ and $w_{k,n+1-k}$ respectively. The system of representatives is parametrized by $C_w(\underline m)$, which is defined as
        \[
            C_w(\underline m)=C_{w^\prime}(\underline m_1)\times C_{w_{k,n+1-k}}(\underline m_2).
        \]
        \end{thm}
        \begin{rem}
            The exact entries of the bottom right $(n+1)\times (n+1)$ block of $\fR$ can be computed using Theorem \ref{thm:1}. This is not needed to compute the sum but will be used in the induction below. The same holds for the entries of $\fC$.
        \end{rem}
        
        \begin{proof} For shorter notation let $M:=N+1-(n+1)$. The factorization $w=w^\prime\circ \begin{pmatrix}
        I_{N+1-(n+1)}&\\&w_{k,n+1-k}
        \end{pmatrix}$ implies
        \begin{equation}
        \label{eq:2}
            \fL\cdot \fC\cdot\fR=L^\prime \cdot C^\prime\cdot R^\prime\cdot \begin{pmatrix}
                I_{M} &\\
                & L
            \end{pmatrix}\cdot \begin{pmatrix}
                I_{M} &\\
                & C
            \end{pmatrix}\cdot \begin{pmatrix}
                I_{M} &\\
                & R
            \end{pmatrix}\ .
        \end{equation}
        By induction or Theorem \ref{thm:1} if $w^\prime$ has only two blocks,  $R^\prime$ has the shape
        \[
            R^\prime=\begin{pmatrix}
                B & V\\
                & I_{n+1}
            \end{pmatrix}.
        \]

        Therefore, the right side of (\ref{eq:2}) is equal to
        \begin{align*}
            &L^\prime \cdot C^\prime\cdot R^\prime\cdot\begin{pmatrix}
                I_{M} &\\
                & LCR
            \end{pmatrix}\\
            &=L^\prime \cdot C^\prime\cdot\begin{pmatrix}
                B &V\cdot(LCR)\\
                & LCR
            \end{pmatrix}\\
            &=L^\prime \cdot C^\prime\cdot\begin{pmatrix}
                I_{M} &\\
                & L
            \end{pmatrix}\cdot \begin{pmatrix}
                I_{M} &\\
                & C
            \end{pmatrix}\cdot\begin{pmatrix}
                B &V\cdot(LCR)\\
                & R
            \end{pmatrix}\\
            &=L^\prime\cdot\Big(C^\prime\begin{pmatrix}
                I_{M} &\\
                & L
            \end{pmatrix}{C^\prime}^{-1}\Big)\cdot C^\prime\cdot\begin{pmatrix}
                I_{M} &\\
                & C
            \end{pmatrix}\cdot\begin{pmatrix}
                B &V\cdot(LCR)\\
                & R
            \end{pmatrix}.
        \end{align*}
        
        Thus,
        \[
            \fR=\begin{pmatrix}
                B & V\cdot(LCR)\\
                & R
            \end{pmatrix},
        \]
        \[
            \fL=L\cdot\Big(C^\prime\begin{pmatrix}
                I_{M} &\\
                & L
            \end{pmatrix}{C^\prime}^{-1}\Big)
        \]
        and
        \[
            \fC=C^\prime\cdot \begin{pmatrix}
                I_{M} &\\
                & C
            \end{pmatrix}.
        \]
        This already proves the claim regarding the central factor $\fC$.
        
        For the left factor first compute that for $1\leq i\leq n+1$
        \[
            C^\prime_{i,M+i}=p^{-\sum_{j=n+1}^{N-1}m_{j,i}}
        \]
        by taking the product over the central factors of the Bruhat decompositions of the $b_{\alpha}(a)$.
        As \[\begin{psmallmatrix}
                I_{M} &\\
                & L
            \end{psmallmatrix}\]
        only has nontrivial entries in the bottom right $(n+1)\times(n+1)$ submatrix, only the above entries of $C^\prime$ are needed to compute that
        \[
            \Big(C^\prime\begin{pmatrix}
                I_{M} &\\
                & L
            \end{pmatrix}{C^\prime}^{-1}\Big)_{i,i+1}=L_{i,i+1}\cdot p^{\sum_{j=n+1}^{N-1}m_{i,j}-m_{i+1,j}}
        \]
        for $1\leq i< n+1$ and that
        \[
            \Big(C^\prime\begin{pmatrix}
                I_{M} &\\
                & L
            \end{pmatrix}{C^\prime}^{-1}\Big)_{i,i+1}=0
        \]
        for $n+1\leq i<N+1$.
        
        Using 
        $
            L^\prime,\ C^\prime\begin{psmallmatrix}
                I_{M} &\\
                & L
            \end{psmallmatrix}{C^\prime}^{-1}\in U(\IQ_p),
        $
        compute
        \[
            \Big(L^\prime+C^\prime\begin{pmatrix}
                I_{M} &\\
                & L
            \end{pmatrix}{C^\prime}^{-1}\Big)_{i,i+1}=L^\prime_{i,i+1}+\Big(C^\prime\begin{pmatrix}
                I_{M} &\\
                & L
            \end{pmatrix}{C^\prime}^{-1}\Big)_{i,i+1},
        \]
        which proves the claims regarding the left factor $\fL$.

        The shape of $\fR$ already proves most of the claims about $\fR$ except 
        \[
            \fR_{i,i+1}=\sum_{j=1}^{k+1}c_{j,n+f}d_{j,k}\cdot p^{\sum_{l=1}^{j-1}m_{l,k}-m_{l,n+f}-m_{j,n+f}}
        \]
        for $i=M$. This entry is given by multiplying the bottom row of $V$ with the first column of $LCR$. The entries of the first column of $LCR$ and by induction also the entries of the bottom row of $W$ are all given by Theorem \ref{thm:1}. Thus,
        \[
            (LCR)_{i,1}=\sum_{j=1}^{n+1}\sum_{l=1}^{n+1} L_{i,j}C_{j,l}R_{l,1}=\sum_{j=1}^{n+1}L_{i,j}C_{j,1}=L_{i,k+1}C_{k+1,1}=d_{i,k}p^{\sum_{l=1}^{i-1}m_{l,k}}
        \]
        and
        \[
            V_{M,i}=c_{i,n+f}p^{-\sum_{l=1}^{i}m_{l,n+f}}
        \]
        as claimed in the theorem.

        Thus, the only thing left is the system of representatives. 
        As in the proof of Theorem \ref{thm:2}, it is enough to show that two different representatives are not equivalent under the right action of $U_w(\IZ_p)$ because the total number of equivalence classes computed in Lemma \ref{lem:2} is the same as the number of representatives proposed in the theorem, and only the right factor of the Bruhat decomposition is relevant. Let 
        \[
            \fR=\begin{pmatrix}
                B & V\cdot LCR\\
                & R
            \end{pmatrix}\quad \text{resp.\ }\Tilde{\fR}=\begin{pmatrix}
                \Tilde{B} & \Tilde{V}\cdot\Tilde{L}\Tilde{C}\Tilde{R}\\
                & \Tilde{R}
            \end{pmatrix}
        \]
    be the right factor for the Bruhat decomposition corresponding to a representative $\underline c$ resp.\ $\Tilde{\underline c}$. First $\underline c_2=\underline c^\prime_2$ will be proven and then $\underline c_1=\underline c^\prime_1$.

    Assume there exists
    \[
        A=\begin{pmatrix}
            A_1 & A_2\\
            & A_3
        \end{pmatrix}\in U_w(\IZ_p),
    \]
    with $A_3\in U_{w_{k,n+1-k}}(\IZ_p)$, such that $\fR\cdot A=\Tilde{\fR}$.

    This implies $R\cdot A_3=\Tilde{R}$ and by Theorem \ref{thm:2} $R=\Tilde{R}$ and $A_3=I_{n+1}$
    and therefore $\underline c_2=\underline c^\prime_2$. 
   Therefore,
    \begin{align*}
        &&\begin{pmatrix}
                B & V\cdot LCR\\
                & R
            \end{pmatrix}\cdot \begin{pmatrix}
            A_1 & A_2\\
            & I_{n+1}
        \end{pmatrix}&=\begin{pmatrix}
                \Tilde{B} & \Tilde{V}\cdot\Tilde{L}\Tilde{C}\Tilde{R}\\
                & R
            \end{pmatrix}\\
        &\Rightarrow& \begin{pmatrix}
                B & V\cdot LCR\\
                & I_{n+1}
            \end{pmatrix}\cdot \begin{pmatrix}
            A_1 & A_2\\
            & I_{n+1}
        \end{pmatrix}&=\begin{pmatrix}
                \Tilde{B} & \Tilde{V}\cdot\Tilde{L}\Tilde{C}\Tilde{R}\\
                & I_{n+1}
            \end{pmatrix}\\
        &\Rightarrow& \begin{pmatrix}
                B & V\\
                & (LCR)^{-1}
            \end{pmatrix}\cdot\begin{pmatrix}
                I_{M}&\\& LCR
            \end{pmatrix}
            \cdot
            \begin{pmatrix}
            A_1 & A_2\\
            & I_{n+1}
        \end{pmatrix}&=\begin{pmatrix}
                \Tilde{B} & \Tilde{V}\\
                & (LCR)^{-1}
            \end{pmatrix}\cdot\begin{pmatrix}
                I_{M}&\\& LCR
            \end{pmatrix}\\  
        &\Rightarrow& \begin{pmatrix}
                B & V\\
                & (LCR)^{-1}
            \end{pmatrix}
            \cdot
            \begin{pmatrix}
            A_1 & A_2\\
            & LCR
        \end{pmatrix}&=\begin{pmatrix}
                \Tilde{B} & \Tilde{V}\\
                & (LCR)^{-1}
            \end{pmatrix}\cdot\begin{pmatrix}
                I_{M}&\\& LCR
            \end{pmatrix}.
    \end{align*}
    Multiplying both sides with
    \[
            \begin{pmatrix}
                I_{M} & \\
                & (LCR)^{-1}
            \end{pmatrix}
    \]
    results in
    \begin{align*}
        &\Rightarrow& \begin{pmatrix}
                B & V\\
                & (LCR)^{-1}
            \end{pmatrix}
            \cdot
            \begin{pmatrix}
            A_1 & A_2\\
            & LCR
        \end{pmatrix}\cdot\begin{pmatrix}
                I_{M} & \\
                & (LCR)^{-1}
            \end{pmatrix}
        &=\begin{pmatrix}
                \Tilde{B} & \Tilde{V}\\
                & (LCR)^{-1}
            \end{pmatrix}\\
        &\Rightarrow& \begin{pmatrix}
                B & V\\
                & (LCR)^{-1}
            \end{pmatrix}
            \cdot
            \begin{pmatrix}
            A_1 & A_2\cdot (LCR)^{-1}\\
            & I_{n+1}
        \end{pmatrix}
        &=\begin{pmatrix}
                \Tilde{B} & \Tilde{V}\\
                & (LCR)^{-1}
            \end{pmatrix}\\
        &\Rightarrow& \begin{pmatrix}
                B & V\\
                & I_{n+1}
            \end{pmatrix}
            \cdot
            \begin{pmatrix}
            A_1 & A_2\cdot (LCR)^{-1}\\
            & I_{n+1}
        \end{pmatrix}
        &=\begin{pmatrix}
                \Tilde{B} & \Tilde{V}\\
                & I_{n+1}
            \end{pmatrix}.
    \end{align*}
    \cite{DR98}*{Proposition 3.1} for $w_{k,n+1-k}$ implies that $LCR\in \GL_{n+1}(\IZ_p)$ and thus
    {$
        \begin{psmallmatrix}
            A_1 & A_2\cdot (LCR)^{-1}\\
            & I_{n+1}
        \end{psmallmatrix}\in U_{w^\prime}(\IZ_p)$}.
    As {$\begin{psmallmatrix}
                B & V\\
                & I_{n+1}
            \end{psmallmatrix}$ and $\begin{psmallmatrix}
                \Tilde{B} & \Tilde{V}\\
                & I_{n+1}
            \end{psmallmatrix}$}
    are the right factors for the Bruhat decomposition corresponding to $\underline c_1$ resp.\ $\Tilde{\underline c_1}$ they are only equivalent under the right action of $U_{w^\prime}(\IZ_p)$ if they and the representatives are equal. Thus, $\underline c_1=\Tilde{\underline c_1}$.

    Combining the two parts gives $\underline c=\Tilde{\underline c}$ and that the representatives are pairwise not equivalent under the right action of $U_w(\IZ_p)$, which concludes the proof.
    \end{proof}

    Knowing a system of representatives and the entries of $L$ and $R$ on the first off-diagonal enables the explicit statement of the Kloosterman sum for general Weyl elements and proves Theorem \ref{thm:4}. Using a modification of the diagrams this parametrized sum can be stated even more compactly.

   The modification will first be explained for Weyl elements with two blocks and then for general Weyl elements. Each edge is replaced by a directed edge pointing downwards or to the right. Two edges are added, one going into $k,k$ and one going out of $1,n$. For $w_0$ the resulting diagram looks as follows:

        
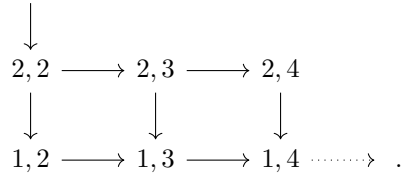
\begin{center}
         \begin{tikzcd}
            \arrow[d]& & & \\
            2,2 \arrow[r] \arrow[d]& 2,3 \arrow[r] \arrow[d] & 2,4 \arrow[d] & \\
            1,2 \arrow[r] & 1,3 \arrow[r] & 1,4 \arrow[r, dotted]&\ .\\
            & & & 
        \end{tikzcd}
        \end{center}

        Here one edge is dotted because it belongs to the right factor while all others belong to the left factor. Further, modify the diagram by adding numbers to the edges. Edges going from $(i,j)$ to $(i,j+1)$ get the number $j+1$, edges going from $(i,j)$ to $(i-1,j)$ get the number $i-1$, the extra edge into $(k,k)$ gets the number $k$ and the dotted edge gets the number $n+1-k$. For $w_0$ this looks as follows:

        \begin{center}
        \begin{tikzcd}
            \arrow[d,"2"]& & & \\
            2,2 \arrow[r,"3"] \arrow[d,"1"]& 2,3 \arrow[r,"4"] \arrow[d,"1"] & 2,4 \arrow[d,"1"] & \\
            1,2 \arrow[r,"3"] & 1,3 \arrow[r,"4"] & 1,4 \arrow[r,"3", dotted]&\ .\\
            & & & 
        \end{tikzcd}
        \end{center}
   For a general Weyl element, the vertices of the diagram are again the $\gamma_i$ but the diagram consists of multiple blocks. Each block corresponds to a factor obtained by repeating the factorization used in the proof of Theorem \ref{thm:3} until all factors are two block Weyl elements, which might not be admissible for $\GL_{N+1}$ but for $\GL_k$ with $k<N+1$. For each such factor, the diagram is modified as before with the addition of $l$ dotted edges between two blocks, where the second block has height $l-1$. The dotted edges to the right of the block corresponding to the factor {$\begin{psmallmatrix}
            I_{N+1-(n+1)}&&\\
            &&I_{k}\\
            &I_{n+1-k}&
        \end{psmallmatrix}$}
    get the number $N+1-k$.
    
    For the Weyl element
    {$
        w_2:=\begin{psmallmatrix}
                &&I_2\\&I_2&\\I_2&&    
            \end{psmallmatrix}=\begin{psmallmatrix}
                &I_4\\I_2&   
            \end{psmallmatrix}\cdot \begin{psmallmatrix}
                I_2\\&&I_2\\&I_2&    
            \end{psmallmatrix}
    $}
    the modified diagram is:

    \begin{center}
    \begin{tikzcd}
            \arrow[d,"4"]& &  &&\\
            4,4 \arrow[r,"5"] \arrow[d,"3"]& 4,5  \arrow[d,"3"] & &&\\
            3,4 \arrow[r,"5"] \arrow[d,"2"]& 3,5  \arrow[d,"2"]\arrow[r,"2", dotted]&\ \arrow[d,"2"]&&\\
            2,4 \arrow[r,"5"] \arrow[d,"1"]& 2,5  \arrow[d,"1"]\arrow[r,"2", dotted]&2,2 \arrow[r,"3"] \arrow[d,"1"]& 2,3  \arrow[d,"1"] & \\
            1,4 \arrow[r,"5"] & 1,5  \arrow[r,"2", dotted]&1,2 \arrow[r,"3"] & 1,3  \arrow[r,"4", dotted]&\ .\\
            & &  &&
    \end{tikzcd}
    \end{center}
    \begin{rem}
        Again, each edge can be thought of as a simple reflection acting. The number is exactly the number of the reflection acting. For the dotted edges this needs to be conjugated by $w$, as in a dotted edge with an $i$ corresponds to $ws_iw^{-1}$ acting. The edges that only have one vertex would send the root to minus itself, for example $w_2s_2w_2^{-1}(\alpha_{3,5})=-\alpha_{3,5}$.
    \end{rem}

    The modified diagram can be used to state the Kloosterman sum in the following way. Let $E_i$ resp.\ $E_i^\prime$ be the set of edges resp.\ dotted edges with the number $i$. A set $E_i$ is ordered from top to bottom and left to right and a set $E_i^\prime$ is ordered from bottom to top. Theorem \ref{thm:4} shows that an edge with the number $i$ corresponds to one summand multiplied by $\psi_i$ resp.\ $\psi_i^\prime$ for dotted edges. The summand for an edge $f$ with number $i$ is $c_{f_1}d_{f_2}p^*$, where $f$ goes from $f_1$ to $f_2$ and $*$ can be computed by
    \[
        *=-m_{f_2}+\sum_{f>e\in E_i} m_{e_1}-m_{e_2}
    \]
    resp.\ 
    \[
        *=-m_{f_1}+\sum_{f>e\in E_i^\prime} m_{e_2}-m_{e_1}
    \]
    for dotted edges. If $f_i$ does not exists, $c_{f_i}$ and $d_{f_i}$ are treated as one.
    This leads to the following Corollary of Theorem \ref{thm:4}.
    \begin{cor}
        \label{cor:2}
        Assume the setting of Theorem \ref{thm:4} and let $E_i$ and $E_i^\prime$ be as defined above. Then
        \begin{align*}
            \Kl_p(\underline m, \psi,\psi^\prime,w)=\sum_{\underline c\in C_w(\underline m)}\Psi\Big(\sum_{i=1}^N &\psi_i(\sum_{e\in E_i}c_{e_1}d_{e_2}p^{-m_{e_2}+\sum_{f\in E_i, f<e}m_{f_1}-m_{f_2}})\\
            +\ &\psi^\prime_i(\sum_{e\in E^\prime_i}c_{e_1}d_{e_2}p^{-m_{e_1}+\sum_{f\in E^\prime_i, f<e}m_{f_2}-m_{f_1}})\Big).
        \end{align*}
    \end{cor}

    \begin{rem}
        In Theorem \ref{thm:4} the sum is ordered by the different blocks and not by the different entries of the left factor because this makes the formula more compact. 
    \end{rem}

\section{Bounds}
\label{sec:4}
\subsection{Preparations}
The main strategy for bounding the now parametrized Kloosterman sum consists of two steps expanding on the ideas used in \cite{BM24}*{Section 5}. In the first step, one realizes that $\Kl_p(\underline m,\psi,\psi^\prime,w)$ resembles nested or interlinked Kloosterman sums for $\GL_2$, to which the Weil bound can be applied. In the seconds step these bounds are combined to achieve a power saving for the whole sum.

To make this concrete look at the sum over a single variable $c_{i,j}$ with $m_{i,j}>0$. This is a $\GL_2$ Kloosterman sum of the form 

\begin{equation}
    \label{eq:8}
    \sum_{c_{i,j}\in(\IZ/C_{i,j}\IZ)^*} \Psi((d_3p^{b_3}+d_4p^{b_4})c_{i,j}+(c_1p^{b_1}+c_2p^{b_2})d_{i,j})
\end{equation}
with $b_1,b_2,b_3$ and $b_4\in\IZ$. Here $c_1,c_2,d_3$ and $d_4$ are the possibly four variables (or their inverses) adjacent to $c_{i,j}$ in the diagram.
Define $A_{i,j}:=v_p((d_3p^{b_3}+d_4p^{b_4}))$ and $B_{i,j}:=v_p((c_1p^{b_1}+c_2p^{b_2}))$.
\begin{rem}
    For ease of notation, the additional factors coming from the characters $\psi$ and $\psi^\prime$ are suppressed by absorbing $|\psi_i|_p^{-1}$ resp.\ $|\psi_i^\prime|_p^{-1}$ into the $b_i$ and ignoring the factor coprime to $p$ as it only rearranges the summands and does not change the size of a single sum.
\end{rem}

Using the Weil bound (\ref{eq:8}) can be bounded by
\[
    |\IZ/(C_{i,j}\IZ)^*|p^{\min (A_{i,j},B_{i,j},0)/2}.
\]
For the following arguments, it is easier to only use $B_{i,j}$ because $d_3$ and $d_4$ in the above sum could simply be zero constantly if the corresponding $m$ is zero. Furthermore, the dependency on $c_1$ and $c_2$ is a problem and working with $b_{i,j}:=\min (b_1,b_2,0)$ instead of $B_{i,j}$ would be easier. It turns out that on average $B_{i,j}$ can be replaced by $b_{i,j}$, which will be shown below. Define $b_{i,j}^{(1)}:=b_1$ (resp.\ $b_{ij}^{(2)}:=b_2$), where $c_1$ is the variable above $c_{i,j}$ in the diagram (resp.\ $c_2$ is the variable to the left in the diagram). 

If $m_{i,j}=0$, the sum over $c_{i,j}$ is no longer a Kloosterman sum and is estimated trivially by $|\IZ/C_{i,j}\IZ|$.

To later combine both estimates define
\[
    b_{i,j}^*:=\begin{cases}
        0 & \text{ if } m_{ij}=0\ ,\\
        b_{i,j} & \text{ else}\ .
    \end{cases}
\]

One could hope for a bound of the form 
\[
    \Kl_p(\underline m,\psi,\psi^\prime,w)\ll |C_w(\underline m)|p^{\sum_{i,j}\frac{1}{2}b_{i,j}^*}.
\]
But as the sums are interlinked, their respective bounds cannot be applied at the same time. To overcome this issue the variables are grouped into two sets such that two variables in the same set are not adjacent in the diagram. For a two-block Weyl element, this simply means grouping all variables $c_{i,j}$ with $i+j$ even (resp.\ odd) together. For more blocks, this stays true within a block but depending on the sizes of the block a different one of the two options might be picked for different blocks.
Define $C_{w}(\underline m)^1$ and $C_{w}(\underline m)^2$ as the sets of values for the $c_{i,j}$ in the respective groups and rearrange the parametrized Kloosterman sum such that the sum over $\underline c_1$ is executed first. Then, the above bound can be applied to each of the $\underline c_1$ sums because they are not interlinked.

To estimate the remaining sum over $\underline c_2$ a generalization of \cite{BM24}*{Lemma 7} is needed ensuring that the previously mentioned dependency of $B_{i,j}$ on the adjacent variables vanishes on average. \cite{BM24}*{Lemma 7} itself is not enough because the variables are grouped into two groups in this paper and not four groups as in \cite{BM24}, which improves the saving by a factor of $2$.

\begin{lemma}
    \label{lem:3}
    Let $k>1$, $e_1,...,e_k\geq0$ and $b_{1},b^\prime_{1},...,b_{k},b^\prime_{k}\in\IZ$. Then 
    \[
        \sum_{c_1=1}^{p^{e_1}}...\sum_{c_k=1}^{p^{e_k}}\prod_{i=1}^{k-1}|c_ip^{b_{i}}+c_{i+1}p^{b^\prime_{i+1}}|_p^{-1/2}\ll \big(\prod_{i=1}^ke_i\big)\cdot p^{\sum_{i=1}^ke_i + \frac{1}{2}\sum_{i=1}^{k-1}\min(b_{i}, b^\prime_{i+1})}.
    \]
\end{lemma}
\begin{rem}
    The factor $\prod_{i=1}^ke_i$ does not appear in \cite{BM24}*{Lemma 7} and is a result of the $c_i$ appearing in two factors on the left side. It can maybe be removed but it is later absorbed into a constant and an $\varepsilon $ in the exponent.
\end{rem}
\begin{proof}
    The lemma follows from the stronger statement
    \[
        \sum_{c_1=1}^{p^{e_1}}...\sum_{c_k=1}^{p^{e_k}}|c_k|_p^{-1/2}\prod_{i=1}^{k-1}|c_ip^{b_{i}}+c_{i+1}p^{b^\prime_{i+1}}|_p^{-1/2}\ll \big(\prod_{i=1}^ke_i\big)\cdot p^{\sum_{i=1}^ke_i + \frac{1}{2}\sum_{i=1}^{k-1}\min(b_{i}, b^\prime_{i+1})},
    \]
    which is proven by induction on $k$.
    The \cite{BM24}*{Proof of Lemma 7} can be seen as a proof for $k=2$ because compared to the sum in \cite{BM24}*{Lemma 7} only a factor $p^{\delta_2/2}$ is added, which makes the $\delta_2$ sum contribute the additional factor $e_2$.

    Assume the statement is true for $k\geq 2$. Separate the $k+1$-th sum from the rest:
    \begin{align*}
        &\sum_{c_1=1}^{p^{e_1}}...\sum_{c_{k+1}=1}^{p^{e_{k+1}}}|c_{k+1}|_p^{-1/2}\prod_{i=1}^{k}|c_ip^{b_{i}}+c_{i+1}p^{b^\prime_{i+1}}|_p^{-1/2}\\
        =& \sum_{c_1=1}^{p^{e_1}}...\sum_{c_{k}=1}^{p^{e_{k}}}\prod_{i=1}^{k-1}|c_ip^{b_{i}}+c_{i+1}^{b^\prime_{i+1}}|_p^{-1/2}\sum_{c_{k+1}=1}^{p^{e_{k+1}}}|c_{k+1}|_p^{-1/2}|c_{k}p^{b_{k}}+c_{{k+1}}p^{b^\prime_{{k+1}}}|_p^{-1/2}.
    \end{align*}
    Splitting the sum over $c_{k+1}$ grouped by $v_p(c_{k+1})$ similar to \cite{BM24}*{Proof of Lemma 7} results in:
    \[
        \sum_{c_{k+1}=1}^{p^{e_{k+1}}}|c_{k+1}|_p^{-1/2}|c_{k}p^{b_{k}}+c_{{k+1}}^{b^\prime_{{k+1}}}|_p^{-1/2}=\sum_{\delta_{k+1}=0}^{e_{k+1}}p^{\delta_{k+1}/2}\sum_{c_{k+1}=1, (p,c_{k+1})=1}^{p^{e_{k+1}-\delta_{k+1}}} |c_{k}p^{b_{k}}+c_{{k+1}}p^{b^\prime_{{k+1}}+\delta_{k+1}}|_p^{-1/2}.
    \]
    If $v_p(c_{k})+b_{k}\neq b^\prime_{{k+1}}+\delta_{k+1}$, the inner sum can be bounded by 
    \begin{align*}
        &\ p^{\delta_{k+1}/2}\sum_{c_{k+1}=1, (p,c_{k+1})=1}^{p^{e_{k+1}-\delta_{k+1}}} |c_{k}p^{b_{k}}+c_{{k+1}}p^{b^\prime_{{k+1}}+\delta_{k+1}}|_p^{-1/2}\\
        \ll& \ p^{e_{k+1}-\delta_{k+1}/2+\frac{1}{2}\min (b^\prime_{{k+1}}+\delta_{k+1}, v_p(c_{k})+b_{k})}\\
        \ll& \ p^{e_{k+1} +\frac{1}{2}\min (b^\prime_{{k+1}},b_{k})}\cdot |c_{k}|_p^{-1/2}.
    \end{align*}
    If $v_p(c_{k})+b_{k}= b^\prime_{{k+1}}+\delta_{k+1}$, the inner sum can be bounded by 
    \begin{align*}
        &\ p^{\delta_{k+1}/2}\sum_{c_{k+1}=1, (p,c_{k+1})=1}^{p^{e_{k+1}-\delta_{k+1}}} |c_{k}p^{b_{k}}+c_{{k+1}}p^{b^\prime_{{k+1}}+\delta_{k+1}}|_p^{-1/2}\\
        =& \ p^{\delta_{k+1}/2+(b^\prime_{{k+1}}+\delta_{k+1})/2}\sum_{\delta=0}^{e_{k+1}-\delta_{k+1}}p^{\delta/2}\sum_{c_{k+1}=1, (p,c_{k+1})=1, p^\delta\mid c_{k+1}+c_{k}}^{p^{e_{k+1}-\delta_{k+1}}} 1\\
        \ll& \ p^{\delta_{k+1}/2+(b^\prime_{{k+1}}+\delta_{k+1})/2}\sum_{\delta=0}^{e_{k+1}-\delta_{k+1}}p^{\delta/2+e_{k+1}-\delta_{k+1}-\delta}\\
        \ll& \ p^{e_{k+1} +\frac{1}{2}\min (b^\prime_{{k+1}},b_{k})}\cdot |c_{k}|_p^{-1/2}.
    \end{align*}
    In both cases, there is an additional factor $e_{k+1}$ from the $\delta_{k+1}$ sum and the rest of the sum resembles the lemma for $k$, which was assumed to be true. 
\end{proof}

To illustrate how to apply Lemma \ref{lem:3} to the sum over $\underline c_2$ look at the previous example:

\begin{center}
\begin{tikzcd}
    \arrow[d]& & & \\
    \textbf{2,2} \arrow[r] \arrow[d]& 2,3 \arrow[r] \arrow[d] & \textbf{2,4} \arrow[d] & \\
    1,2 \arrow[r] & \textbf{1,3} \arrow[r] & 1,4 \arrow[r, dotted]&\ .\\
    & & & 
\end{tikzcd}
\end{center}
        
Here the variables belonging to $\underline c_1$ are bold. After estimating their sums as above, the following dependencies are left:

\begin{center}
\begin{tikzcd}
    \arrow[d]& & & \\
    \textbf{2,2} & 2,3 \arrow[r] \arrow[d] & \textbf{2,4} & \\
    1,2 \arrow[r] & \textbf{1,3} & 1,4 \arrow[r, dotted]&\ .\\
    & & & 
\end{tikzcd}
\end{center}

The sum over $c_{1,2}$ and $c_{2,3}$ fits the shape of Lemma \ref{lem:3} and the sum over $c_{1,4}$ does as well. In general, the sums over variables being in the same diagonal from bottom left to top right will be in the shape of Lemma \ref{lem:3} and independent of the other sums. If one of the sums over a bold variable was estimated trivially because the corresponding $m$ is zero, the diagonal can be split into two parts and Lemma \ref{lem:3} can be applied to both parts separately.

The total bound is the product of the bounds for the sums over $\underline c_1$ and the trivial bounds for the sums over $\underline c_2$. The same process can be repeated with $\underline c_1$ and $\underline c_2$ swapped. Combining both bounds using the geometric mean results in the following lemma. 
\begin{lemma}
    For all $\varepsilon>0$ the bound
    \[
        \Kl_p(\underline m,\psi,\psi^\prime,w)\ll_{\varepsilon} |C_w(\underline m)|^{1+\varepsilon}p^{\frac{1}{4}\sum_{i,j}b_{i,j}^*}
    \]
    holds.
\end{lemma}

\begin{proof}
    The only thing left to prove is that the product over the $e_i$ can be absorbed into the $\varepsilon$ and the implicit constant. Corollary \ref{cor:1} implies
    that each $e_i$ is bounded by $\sum_{i,j} m_{i,j}$. Thus, the product of the $e_i$'s is bounded by
    \[
        \big(\sum_{i,j} m_{i,j}\big)^{n^2}\ll _\varepsilon\big(p^{\sum_{i,j} m_{i,j}}\big)^{\varepsilon}\ll |C_w(\underline m)|^{\varepsilon}.
    \]
\end{proof}

\subsection{Proof of Theorem \ref{thm:5}}
\label{sec:4.2}

Relating $\sum_{i,j}b_{i,j}^*$ and $\sum_{i,j}|j-i+1|m_{i,j}$ by a factor proves the desired bound because a statement like
\begin{equation}
\label{eq:3}
    -\sum_{i,j}b_{i,j}^*\geq C\sum_{i,j}|j-i+1|m_{i,j}
\end{equation}
implies 
\[
    p^{\frac{1}{4}\sum_{i,j}b_{i,j}^*}\leq |C_w(\underline m)|^{4C}
\]
and thus
\[
    \Kl_p(\underline m,\psi,\psi^\prime,w)\ll_{\varepsilon} |C_w(\underline m)|^{1+\varepsilon-4C}.
\]

Theorem \ref{thm:4} claims $C=\frac{1}{l(w)}$, where one additional factor $2$ comes from having two sets of variables and the other from the Weil bound.

For now assume $m_{i,j}>0$ for all $i,j$. The case $m_{i,j}=0$ will be discussed later. The main idea to prove a statement like (\ref{eq:3}) is to estimate $b_{i,j}$ by $b_{i,j}^{(1)}$ or $b_{i,j}^{(2)}$ and group the $b_{i,j}$'s together such that their sum can be bounded by a sum of $m_{i,j}$'s. These groups are roughly going to be the downward diagonals in the diagram. They are motivated by Corollary \ref{cor:2} as the $p$-power of edges in a downward diagonal sum up the desired sums of $m_{i,j}$'s.

For a Weyl element for $\GL_{N+1}$ define $N$ sets of $b_{i,j}$'s, where the same $b_{i,j}$ can be contained in multiple sets. Let $B_k$ denote the $k$-th set. To illustrate how these sets are defined the previous example

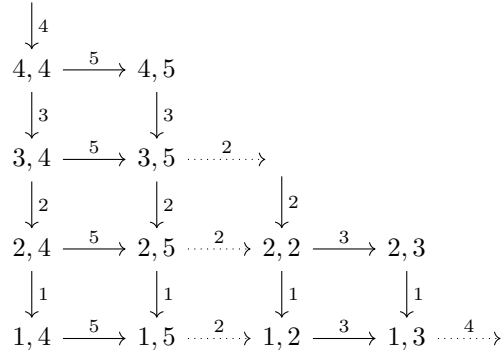
\begin{center}
\begin{tikzcd}
        \arrow[d,"4"]& &  &&\\
        4,4 \arrow[r,"5"] \arrow[d,"3"]& 4,5  \arrow[d,"3"] & &&\\
        3,4 \arrow[r,"5"] \arrow[d,"2"]& 3,5  \arrow[d,"2"]\arrow[r,"2", dotted]&\ \arrow[d,"2"]&&\\
        2,4 \arrow[r,"5"] \arrow[d,"1"]& 2,5  \arrow[d,"1"]\arrow[r,"2", dotted]&2,2 \arrow[r,"3"] \arrow[d,"1"]& 2,3  \arrow[d,"1"] & \\
        1,4 \arrow[r,"5"] & 1,5  \arrow[r,"2", dotted]&1,2 \arrow[r,"3"] & 1,3  \arrow[r,"4", dotted]&\ \\
        & &  &&
\end{tikzcd}
\end{center}

corresponding to $w_2$ with $N=5$ is used.

The dotted arrows are ignored in the following definition. There are two kinds of sets. The first one corresponds to the numbers $k$ such that there is no horizontal edge with that number.

For such a $k$ start with selecting $b_{k,j}$ such that $k,j$ is the furthest to the right in the $k,*$-row in the diagram. If $b_{i,j}$ was just selected, select the next element in the $k$-th set by the following rules. If the diagram has an entry $a,b$ to the top left of $i,j$, add $b_{a,b}$. If not, add $b_{i,i},b_{i,i+1},..., b_{i,j-1}$ and $b_{j+1,l}$ such that $j+1,l$ is the one furthest to the right in the $j+1,*$ row, then repeat until the left column of the diagram is reached. 

The sets of the second kind correspond to $k$ such that there is a horizontal edge with the number $k$. Start by adding $b_{1,k}$ to the set. After adding $b_{i,j}$ add the $b_{a,b}$, where $a,b$ is to the top left of $i,j$ in the diagram, if $a,b=i+1,j-1$ or in other words if they belong to the same block. If not, $b_{i,j}$ is at an edge of a block. If it is the top edge, add $b_{j+1,l}$ such that $j+1,l$ is the one furthest to the right in the $j+1,*$ row and continue as in the first kind of sets. If it is the left edge, add $b_{i,j},..., b_{j,j}$ and then continue as in the first kind of sets. 

In the example, the five sets in order would look like this
\[
    \{b_{1,3},b_{2,2},b_{3,5},b_{4,4}\}, \{b_{2,3},b_{2,2},b_{4,5},b_{4,4}\}, \{b_{1,3},b_{2,2},b_{3,5},b_{4,4}\}, \{b_{4,5},b_{4,4}\}, \{b_{1,5},b_{1,4},b_{2,4},b_{3,4},b_{4,4}\}.
\]

A sum over each group is estimated by estimating the variables in the sets of the first kind by $b_{i,j}^{(1)}$ except for the ones on the top edge of a block that do not have an incoming edge from above, which are estimated by $b_{i,j}^{(2)}$, and the variables of the second set which are added until we hit an edge by $b_{i,j}^{(2)}$ and the ones after as in the first set. 

The sets are defined this way to achieve the following bounds. For a set of the first kind
\begin{equation}
\label{eq:5}
    \sum_{b_{i,j}\in B_k} b_{i,j}\leq -\sum_{j} m_{k,j}
\end{equation}
holds and for a set of the second kind
\begin{equation}
    \label{eq:6}
    \sum_{b_{i,j}\in B_k} b_{i,j}\leq -\sum_{j} m_{k,j}-\sum_{i} m_{i,k}
\end{equation}
holds.

These estimates can be combined to achieve a bound like (\ref{eq:3}) starting with a block with $a$ rows and $b$ columns:
\[
    \sum_{k=1}^a (a+1-k)\sum_{b_{i,j}\in B_k} b_{i,j} + \sum_{k=a+1}^{a+b-1} (k-a)\sum_{b_{i,j}\in B_k} b_{i,j}\leq -\sum_{1\leq i\leq a\leq j\leq a+b-1} (j-i+1)m_{i,j},
\]
where $m_k:=(a+1-k)$ resp.\ $m_k:=(k-a)$ is the multiplicity of $B_k$. For a Weyl element with more than two blocks, the multiplicity of $B_k$ is the maximum of the multiplicities of $B_k$ for the blocks. Each $b_{i,j}$ appears at most $l(w)$ many times in the sum over the $B_k$ with multiplicities. For one block, this is a direct computation, and for multiple blocks, it can be bounded by the sum of the sizes of the blocks because of the way the multiplicities of the $B_k$ were defined. This argument will be made more precise below. This results in
\begin{equation}
\label{eq:4}
    l(w)\cdot \sum_{i,j}b_{i,j}\leq \sum_{k=1}^N m_k\sum_{b_{i,j}\in B_k} b_{i,j}\leq -\sum_{i,j}(j-i+1)m_{i,j}.
\end{equation}

To deal with $m_{i,j}=0$ the construction of the $B_k$ is modified according to the following three cases

If in the construction of a set of the second kind before crossing the first block boundary a $b_{i,j}$ with $m_{i,j}=0$ which is estimated by $b_{i,j}^{(2)}$ would be added, instead add $b_{i+1,j}$, estimate it by $b_{i+1,j}^{(2)}$ and continue the construction from there, or if $b_{i+1,j}$ does not exist, add $b_{j,l}$ such that $j,l$ is the furthest to the right in the $j,*$-row, estimate it by $b_{j,l}^{(1)}$ and continue as after the top edge case for sets of the second kind.

If in the construction of a set a $b_{i,j}$ with $m_{i,j}=0$ which is estimated by $b_{i,j}^{(2)}$ and which is not covered by the previous case would be added, instead add $b_{j,l}$ such that $j,l$ is the furthest to the right in the $j,*$-row, estimate it by $b_{j,l}^{(1)}$, but continue as if $b_{i,j}$ was added. 

If in the construction of a set a $b_{i,j}$ with $m_{i,j}=0$ which is estimated by $b_{i,j}^{(1)}$ would be added, instead add the variable $b_{i,l}$ to the left of it ($b_{i,j-1}$ in most cases), estimate it by $b_{i,l}^{(1)}$, but continue as if $b_{i,j}$ was added. 

In all three cases, if the new candidate also has $m_{i,j}=0$, it is replaced following the same rules.

Claim: The new sets $B_k$ still satisfy both inequalities in (\ref{eq:4}) even with $b_{i,j}^*$ instead of $b_{i,j}$.

\textit{Proof of claim:} The second claim is clear once the first is proven because only $b_{i,j}$ with $m_{i,j}>0$ are added to the $B_k$.

The first inequality holds because $b_{i,j}$ can only appear at most once in $B_k$ if there is a variable $k,b$ on or to the right of the downwards diagonal starting at $i,j$ in the diagram, or if $B_k$ is of the second kind and there is a variable $a,k$ on or to the left of the downwards diagonal starting at $i,j$ in the diagram. The maximal sum over these multiplicities can be computed inductively over the number of blocks and is exactly $l(w)$.

The seconds inequality holds because if $b_{i,j}$ is replaced by $b_{a,b}$ the estimate for $b_{a,b}$ is smaller than the estimate for $b_{i,j}$ under the assumption that $m_{i,j}=0$ and thus (\ref{eq:5}) and (\ref{eq:6}) still hold, which proves the claim.

Thus, (\ref{eq:3}) holds with $C=\frac{1}{l(w)}$ and there are only two more things left to prove Theorem \ref{thm:5}.
The first one is the dependency on the characters. Each estimate of a single Kloosterman sum in the discussion at the start of this section gets worse by at most a factor of $\max_{1\leq j\leq n}\max(|\psi_j|_p^{-1/2},|\psi^\prime_j|_p^{-1/2})$ and $l(w)/2$ of them are used at the same time. Therefore, the total error is at most a factor $\max_{1\leq j\leq n}\max(|\psi_j|_p^{-1/2},|\psi^\prime_j|_p^{-1/2})^{l(w)/2}$. More on this in the next chapter.

The second is that $\Kl_p(\underline m,\psi,\psi^\prime,w)$ was bounded and not $\Kl_p(\psi,\psi^\prime,\fn)$ but the number of possible $\underline m$ for a fixed $\fn$ and thus a fixed $\lambda$ can be absorbed into the $\varepsilon$ in the exponent. This concludes the proof of Theorem \ref{thm:5} for $\GL_{N+1}(\IZ_p)$.\qed

\begin{rem}
    Heuristically one would expect the case with all $m_{ij}$ of the same or at least similar size to be the worst case. Surprisingly the same method as above can prove a saving of $\frac{1}{2(N-1)}$ for the long Weyl element with all $m_{ij}$ being equal. This could be seen as in indication that the optimal saving should be $\frac{1}{O(N)}$ but it could also be that there are unforeseen edge cases for $\GL_{N+1}$ that produce bigger sums.
\end{rem}

\subsection{Character dependency}
\label{subsec:43}
The character dependency in Theorem \ref{thm:5} might seem to generous on first sight especially with the $\GL_2$ picture in mind. The purpose of this chapter is to argue that the dependency in general can be as bad as $|\psi|_p^{\Theta(N^2)}$ and to prove Theorem \ref{thm:6} giving a result with a lighter dependency but a worse power saving.

\begin{proof}[Proof of Theorem \ref{thm:6}] Using the same notation and ideas as before, let $m:=\max_{i,j} m_{ij}$. Then there is a $k$ s.t. 
\[
    -\sum_{b_{ij}\in B_k} b_{i,j}\geq m\ .
\]
As $|B_k|\leq N$, there is a pair $i,j$ s.t. the saving in the $c_{i,j}$ is at least $p^{-b_{i,j}/2}\geq p^{\frac{m}{2N}}$. Using $\sum_i^n r_i\leq l(w)\cdot m$ this saving is at least $p^{\frac{1}{2N\cdot l(w)}\sum_i^n r_i}$. As this is only considering one sum, this can only be affected by one character and because the $p$-power of an edge affected by $\psi_i$ is at least $p^{-r_i}$, the lost saving is bounded by $p^{r_i/2}$.
\end{proof}

In the $\GL_2$-case the dependency on the character can be understood through the following equality
\[
    S(m,n,c)=p^{-1}S(pm,pn,pc)\ ,
\]
which indicates that assuming a power saving of $\alpha$ multiplying $m$ and $n$ with $p$ increases the sum by $p^{\alpha}$. 
To get a grasp on the same effect in general let $w_l$ be the long Weyl element for $\GL_{N+1}$ and assume all the $m_{i,j}$ are at least $1$. 

Then for $1\leq k\leq N$ the following holds
\[
    \Kl_p(\underline m,\psi,\psi^\prime,w) = p^{-(N+1-k)\cdot k} \Kl_p(\Tilde{\underline m},\Tilde{\psi},\Tilde{\psi}^\prime,w)
\]
for
\[
    \Tilde{m}_{ij}=\begin{cases} m_{ij}+1 & \text{ if } |i-j|=k-1\\
    m_{ij} &\text{ else ,}
    \end{cases}\text{, }\Tilde{\psi}_i = \begin{cases}p\cdot \psi_i &\text{ if }i=k\\
    \psi_i &\text{ else}
    \end{cases} \text{ and } \Tilde{\psi}^\prime_i = \begin{cases}p\cdot \psi^\prime_i &\text{ if }i=N+1-k\\
    \psi^\prime_i &\text{ else .}
    \end{cases}
\]
This indicates assuming a power saving of $\alpha$ that in the worst case $i=(N+1)/2$ multiplying $\psi_i$ and $\psi^\prime_{N+1-i}$ by $p$ increases the sum by $p^{\alpha (N+1)^2/4}$. It also shows that not all $\psi_i$ have the same effect. Furthermore, multiplying all $\psi_i$ and $\psi^\prime_i$ by $p$ and increasing all the $m_{ij}$ by one increases the sum by a factor of $p^{(N^3+3N^2+2N)/6}$. This implies that any uniform power saving of the form $\frac{1}{o(N^3)}$ as for example the one in Theorem \ref{thm:6} has to have a character dependency that also depends on $N$. Only using this example one could still hope to decrease the exponent $l(w)/2$ to something linear in $N$ but for other Weyl elements or $m_{ij}=0$ the picture becomes less clear.

\subsection{Congruence Subgroups}
In many instances, the Kloosterman sums appearing in applications are not over the whole group $\GL_{N+1}(\IZ_p)$ but over a congruence subgroup, see for example \cites{AB24,Bl23,Mi24}.
In this section, Theorem \ref{thm:5} and Theorem \ref{thm:6} are proven for the congruence subgroup $\Gamma_0(q)$, which is defined as
\[
    \Gamma_0(q):=\{\gamma\in\GL_{N+1}(\IZ_p): \gamma\equiv \begin{psmallmatrix}\gamma^\prime & v\\
    0 & *
    \end{psmallmatrix} \mod q, \text{ for a } N\times N\text{ matrix } \gamma^\prime\},
\]
where $q=p^l$.
\begin{proof}
    The fixed choice of a reduced expression for the Weyl elements makes it easier to work with a modified congruence subgroup and then translate the result. For a matrix $A$ define $A^{T_w}:=(w_lAw_l)^T$, where $w_l$ is the long Weyl element, and write
    \[
        \Gamma := \{A^{T_w}:A\in\Gamma_0(q)\}.
    \]
    Note the following identities $(A^{T_w})^{T_w}=A$, $(AB)^{T_w}=B^{T_w}A^{T_w}$, $U^{T_w}=U$, $U_w^{T_w}=U_{w^{-1}}$, and $w^{T_w}=w$ for an admissible Weyl element $w$ resulting in
    \[
        U(\IZ_p)\backslash \big(U(\IQ_p)\fc wU(\IQ_p)\cap \Gamma_0(q)\big)/U_w(\IZ_p)=\Big(U_{w^{-1}}(\IZ_p)\backslash \big(U(\IQ_p)w\fc^{T_w}U(\IQ_p)\cap \Gamma\big)/U(\IZ_p)\Big)^{T_w}.
    \]
    For a character $\psi$ of $U(\IQ_p)$ trivial on $U(\IZ_p)$ define $\psi^{T_w}$ by
    \[
        \psi^{T_w}(A):=\psi(A^{T_w}),
    \]
    which is well-defined because $\psi(AB)=\psi(BA)$.

    For any admissible Weyl element $w$ and diagonal element $\fc$ the Kloosterman sums over the two congruence subgroups can be related through
    \[
        \Kl^{\Gamma_0(q)}_p(\psi,\psi^\prime, \fc\cdot w)={\Kl^\prime}^\Gamma_p((\psi^\prime)^{T_w},\psi^{T_w}, w\fc^{T_w}w^{-1}\cdot w)={\Kl}^\Gamma_p(\overline\psi^{T_w},(\overline \psi^\prime)^{T_w}, \fc^{-T_w}\cdot w^{-1}),
    \]
    where the second equality follows from Remark \ref{rem:1}.
    
    Because $\fc^{-T_w}$, $\overline\psi^{T_w}$ and $(\overline \psi^\prime)^{T_w}$ only differ from $\fc$, $\psi$, $\psi^\prime$ by a permutation of the $r_k$ in the exponent vector and a permutation and sign change of the $\psi_i$ resp.\ $\psi_i^\prime$ and $l(w)=l(w^{-1})$, the bounds claimed in Theorem \ref{thm:5} and Theorem \ref{thm:6} remain the same for the new sums and it is enough to prove the theorem for $\Gamma$ instead of $\Gamma_0(q)$.

    The unique reduced expression for an admissible Weyl element was defined such that there is exactly one factor $s_1$ in this expression. The factors $b_\alpha(a)$ satisfy 
    \[
        b_{\alpha_i}(a)\in\Gamma
    \]
    for all $i>1$ and
    \[
        b_{\alpha_1}(a)\in\Gamma \iff \mu(a)\geq l.
    \]
    Thus, the factors of $b_{\underline\alpha}(a)$ except one are always in $\Gamma$ and $U(\IZ_p)\subset\Gamma$. Therefore, $b_{\underline\alpha}(a)\in\Gamma$ iff the factor corresponding to $s_1$ is in $\Gamma$. Let $\alpha_{1k}$ be the corresponding element in $R(w^{-1})$. Then $b_{\underline\alpha}(a)\in\Gamma$ iff $m_{1k}\geq l$. This implies
    \[
        \Kl^{\Gamma}_p(\psi,\psi^\prime, \fc\cdot w)=\sum_{\underline m\in\CM_{w}(r), m_{1k}\geq l}\Kl_p(\underline m,\psi,\psi^\prime,w)
    \]
    and it can be bounded using the bounds obtained in Section \ref{sec:4.2}, which concludes the proof of Theorem~\ref{thm:5}.
\end{proof}


\end{document}